\documentclass[11pt]{amsart}

\usepackage[T1]{fontenc}
\usepackage[utf8]{inputenc}
\usepackage{lmodern}
\usepackage{amsmath,amssymb,amsthm,mathtools}
\usepackage{enumitem}
\usepackage[
 pdfauthor={},
 pdftitle={},
 pdfcreator={},
 colorlinks=true,
 linkcolor=blue,
 citecolor=blue,
 urlcolor=blue
]{hyperref}

\theoremstyle{plain}
\newtheorem{theorem}{Theorem}[section]
\newtheorem{proposition}[theorem]{Proposition}
\newtheorem{corollary}[theorem]{Corollary}
\newtheorem{lemma}[theorem]{Lemma}

\theoremstyle{definition}
\newtheorem{definition}[theorem]{Definition}
\newtheorem{example}[theorem]{Example}

\theoremstyle{remark}
\newtheorem{remark}[theorem]{Remark}

\DeclareMathOperator{\Spec}{Spec}
\DeclareMathOperator{\Supp}{Supp}

\DeclareMathOperator{\Ass}{Ass}
\DeclareMathOperator{\Hom}{Hom}
\DeclareMathOperator{\End}{End}

\DeclareMathOperator{\Tr}{Tr}
\DeclareMathOperator{\Max}{Max}
\DeclareMathOperator{\Soc}{Soc}

\newcommand{\ann}{\operatorname{ann}}

\newcommand{\FGId}{\operatorname{FId}}

\newcommand{\GammaSD}{\Gamma_{\mathrm{sd}}}

\numberwithin{equation}{section}

\title[Annihilator Multiplication Modules]
{Annihilator Multiplication Modules: Ring Characterizations and Constructions}

\author[Kim]{Hwankoo Kim}
\address{Division of Computer Engineering, Hoseo University, Asan, Republic of Korea}
\email{hkkim@hoseo.edu}

\author[Ko\c{c}]{Suat Ko\c{c}}
\address{Department of Mathematics, Marmara University, Istanbul, Turkiye}
\email{suat.koc@marmara.edu.tr}

\date{September 11, 2026}

\begin{document}

\begin{abstract}
An $A$-module $E$ is annihilator multiplication if the annihilator of each
element equals that of $IE$ for a finitely generated ideal $I$.
We characterize rings for which every faithful module has this property as
the commutative quasi-Frobenius rings, and rings for which every module has
the property as the Artinian principal ideal rings. Both characterizations
reduce to two-generated modules. Over a reduced ring with finitely many
minimal primes, faithful annihilator multiplication modules are exactly the
regular-torsion-free modules, and principal ideals suffice in the definition.
An ascending chain condition on annihilator ideals gives finite detection
of the module annihilator, yielding localization and graph rigidity results.
We also prove equality of associated primes with those of the faithful
quotient and establish projective tensor and trace characterizations.
Over local square-zero rings, the property is equivalent to nonsingularity
of a bilinear multiplication map. A projective dimension argument gives a
sharp length bound, attained by explicit faithful indecomposable
nonprojective modules whose endomorphism rings are computed.
Every cyclic submodule of these examples embeds in the ring, although the
whole module is not torsionless.
Support, hereditary torsion, and amalgamation criteria connect these
structural results with further module constructions.
\end{abstract}

\maketitle

\section{Introduction and basic reductions}

Annihilators give a direct way to measure how a ring acts on a module. For an
$A$-module $E$ and an element $e\in E$, the ideal $\ann_A(e)$ records all ring
elements that kill $e$. The class of annihilator multiplication modules asks
for these elementwise annihilators to be represented by annihilators of
submodules of the special form $IE$, where $I$ is a finitely generated ideal of
the base ring. Thus the condition replaces arbitrary cyclic data inside $E$ by
finitely generated ideal data in $A$.

The concept was introduced in the study of Baer modules
\cite{JTK22} and investigated systematically in \cite{KocAMM}.
A central question is how a condition on individual element annihilators
constrains the ring, without imposing finite generation on the module.
The cyclic embedding criterion of Proposition~\ref{prop:cyclic-embedding}
provides a useful starting point: over the faithful quotient, every cyclic
submodule must embed in some finite free module.

The first main classification, Theorem~\ref{thm:faithful-qf}, removes the
reducedness assumption from \cite[Theorem~5]{KocAMM}.
Every faithful module is annihilator multiplication precisely over a
commutative quasi-Frobenius ring. Requiring the condition for all modules
is strictly stronger: Theorem~\ref{thm:all-modules-pir} identifies the
Artinian principal ideal rings. In each case it suffices to test particular
two-generated modules. The double-annihilator characterization of
quasi-Frobenius rings is classical \cite{Lam99, Storrer69}; the results here
establish the annihilator multiplication characterizations and their
consequences for closure under module constructions.

Two further structural results treat modules individually.
Theorem~\ref{thm:finite-detection-acc} shows that an ascending chain condition
on annihilator ideals of the faithful quotient forces a finitely generated
submodule to have the same annihilator as the entire module.
This gives a localization theorem without finite generation of the ambient
module and the exact Noetherian local-to-global criterion of
Theorem~\ref{thm:noetherian-local-global}.
Theorem~\ref{thm:reduced-finite-min} identifies the faithful modules over
reduced rings with finitely many minimal primes: they are precisely the
modules on which regular elements act injectively. In this setting,
principal ideals suffice to represent all element annihilators.

The associated-prime theorem from the preceding framework is also
strengthened. Ko\c{c} proved $\Ass_A(E)\subseteq\Ass(A)$ for faithful
annihilator multiplication modules, and equality when $E$ contains an
element with zero annihilator \cite[Theorem~6]{KocAMM}.
Theorem~\ref{thm:associated-primes} proves that faithfulness alone suffices,
answering \cite[Question~1]{KocAMM}. We retain faithfulness; it is the
additional existence of an element with zero annihilator that is unnecessary.
For projective constructions, Theorems~\ref{thm:hom-trace-projective}
and \ref{thm:tensor-trace} identify the relevant property with that of an
idempotent trace summand. In particular, tensoring by any finitely
generated projective module preserves the property.

The square-zero results supply examples beyond these projective constructions.
Theorem~\ref{thm:squarezero-bilinear} turns annihilator multiplication into
nonsingularity of a bilinear map from the maximal ideal and the module
modulo its socle. Applying the classical projective dimension theorem
produces a lower bound on the socle dimension.
Theorem~\ref{thm:sharp-indecomposable-family} constructs modules attaining
the bound, proves that they are faithful, indecomposable, and nonprojective,
and computes their endomorphism rings.
Corollary~\ref{cor:cyclic-not-torsionless} further shows that every
cyclic submodule can embed in the ring even when homomorphisms from the
whole module to the ring fail to separate its elements.

Section~\ref{sec:structural} establishes the ring and module classifications.
Section~\ref{sec:support} treats support, projective constructions, and
associated primes. Section~\ref{sec:graph} studies the support-disjointness
graph, the annihilator semilattice, and hereditary torsion theories.
The distinction between annihilators of finitely generated submodules and
annihilators of finitely generated ideal multiples is retained;
Proposition~\ref{prop:lattice-quotient} gives the exact relationship, while
Theorem~\ref{thm:graph-rigidity} and the finite-minimal-prime formulas describe
its graph-theoretic consequences. Section~\ref{sec:amalgamation} treats
amalgamations and the square-zero constructions.

Throughout, rings are commutative with identity, ring homomorphisms preserve
the identity, and modules are unital. The zero module is allowed in
module constructions and satisfies the annihilator multiplication condition.
Whenever prime modules or support-disjointness graphs are considered, the
module is explicitly assumed to be nonzero. For an $A$-module $M$, write
$\Supp_A(M)=\{\mathfrak p\in\Spec(A):M_{\mathfrak p}\ne0\}$ and
$\ann_A(X)=\{a\in A:aX=0\}$ for a subset $X$ of $M$.
For ideals $B,I$ of $A$, put $(B:I)=\{a\in A:aI\subseteq B\}$.
We repeatedly use the elementary identity
\begin{equation}\label{eq:colon-annihilator}
\ann_A(IM)=(\ann_A(M):I).
\end{equation}
In particular, modules with the same annihilator have ideal multiples with
the same annihilator, for every ideal $I$. Subscripts are omitted when the
base ring is clear.

\begin{definition} \cite{KocAMM}
An $A$-module $E$ is called an \emph{annihilator multiplication module} if, for
every $e\in E$, there exists a finitely generated ideal $I$ of $A$ such that
$\ann_A(e)=\ann_A(IE).$
\end{definition}

We begin with three standard reductions that will be used throughout. The first
extends the defining condition from elements to finitely generated submodules;
see \cite[Proposition 1(i)]{KocAMM}.

\begin{lemma}
\label{lem:fg-submodule}
Let $E$ be an annihilator multiplication $A$-module. If $L$ is a finitely
generated submodule of $E$, then there exists a finitely generated ideal $I$ of
$A$ such that $\ann_A(L)=\ann_A(IE).$
\end{lemma}

\begin{proof}
Write $L=Ax_1+\cdots+Ax_r$ and choose finitely generated ideals $I_i$ with
$\ann_A(x_i)=\ann_A(I_iE)$. The ideal $I=I_1+\cdots+I_r$ is finitely
generated and
\[
\ann_A(L)=\bigcap_{i=1}^r\ann_A(x_i)
=\bigcap_{i=1}^r\ann_A(I_iE)=\ann_A(IE).
\]
For $L=0$, take $I=0$.
\end{proof}

The class is also stable under finite direct powers; see
\cite[Proposition~5]{KocAMM}.

\begin{lemma}
\label{lem:direct-powers}
If $E$ is an annihilator multiplication $A$-module, then $E^n$ is an
annihilator multiplication $A$-module for every integer $n\geq 1$.
\end{lemma}

\begin{proof}
For $x=(x_1,\ldots,x_n)\in E^n$, choose finitely generated ideals $I_i$ with
$\ann_A(x_i)=\ann_A(I_iE)$ and put $I=I_1+\cdots+I_n$.
Then $\ann_A(x)=\ann_A(IE)=\ann_A(I E^n)$, as required.
\end{proof}

The next descent principle will be useful when an annihilator multiplication
module is replaced by a submodule with the same annihilator; see
\cite[Corollary 2(ii)]{KocAMM}.

\begin{lemma}
\label{lem:equal-ann-submodule}
Let $E$ be an annihilator multiplication $A$-module, and let $L$ be a submodule
of $E$ such that $\ann_A(L)=\ann_A(E).$
Then $L$ is an annihilator multiplication $A$-module.
\end{lemma}

\begin{proof}
For $x\in L$, choose a finitely generated ideal $I$ with
$\ann_A(x)=\ann_A(IE)$. Since $\ann_A(E)=\ann_A(L)$,
\eqref{eq:colon-annihilator} gives $\ann_A(IE)=\ann_A(IL)$.
\end{proof}

We shall also use the following elementary change-of-rings observation. It is
included to make the later reduction in the trivial amalgamated case explicit.

\begin{lemma}
\label{lem:quotient-change}
Let $\pi:R\to A$ be a surjective homomorphism of commutative rings, and let $E$
be an $A$-module, viewed as an $R$-module through $\pi$. Then $E$ is an
annihilator multiplication $A$-module if and only if it is an annihilator
multiplication $R$-module.
\end{lemma}

\begin{proof}
For every subset $X\subseteq E$ one has $\ann_R(X)=\pi^{-1}(\ann_A(X)).$
Assume first that $E$ is annihilator multiplication over $A$. Given $x\in E$,
choose a finitely generated ideal $I=(a_1,\ldots,a_n)$ of $A$ such that
$\ann_A(x)=\ann_A(IE)$. Let $r_i\in R$ be a lift of $a_i$, and set
$H=(r_1,\ldots,r_n)$. Then $H$ is finitely generated, $HE=IE$, and hence
\[
\ann_R(x)=\pi^{-1}(\ann_A(x))= \pi^{-1}(\ann_A(IE))=\ann_R(HE).
\]
Thus $E$ is annihilator multiplication over $R$.

Conversely, assume that $E$ is annihilator multiplication over $R$. For
$x\in E$, choose a finitely generated ideal $H$ of $R$ with
$\ann_R(x)=\ann_R(HE)$. Put $I=\pi(H)$, which is a finitely generated ideal of
$A$. Since $HE=IE$, we have $\pi^{-1}(\ann_A(x))=\pi^{-1}(\ann_A(IE)).$
The map $\pi$ is surjective, so equality of inverse images implies
$\ann_A(x)=\ann_A(IE)$. Therefore $E$ is annihilator multiplication over $A$.
\end{proof}

\section{Cyclic embeddings and ring-theoretic characterizations}
\label{sec:structural}

Passing to the faithful quotient separates the annihilator multiplication
condition from the particular presentation of the module. This section
uses that reduction to obtain cyclic embedding criteria, finite detection
of the module annihilator, and two different classifications of rings.
The distinction between faithful modules and arbitrary modules is essential.

\subsection{Cyclic embeddings and finite detection}

The relevant embedding condition is imposed on each cyclic submodule
separately; it does not assert that the entire module embeds in a finite
free module.

\begin{proposition}\label{prop:cyclic-embedding}
Let $E\ne0$ be an $A$-module and put $R=A/\ann_A(E)$.
Then $E$ is annihilator multiplication if and only if every cyclic
$R$-submodule of $E$ embeds in $R^n$ for some positive integer $n$.
If $R$ is Noetherian, these conditions are also equivalent to
\[
 \ann_R(x)=\ann_R\bigl(\ann_R(\ann_R(x))\bigr)
 \qquad\text{for every }x\in E.
\]
\end{proposition}

\begin{proof}
By Lemma~\ref{lem:quotient-change}, we may work over $R$, where $E$ is
faithful. For $J=\ann_R(x)$, the defining condition is
$J=\ann_R(a_1,\ldots,a_n)$ for finitely many $a_i\in R$.
This equality holds precisely when
\[
 R/J\longrightarrow R^n,\qquad r+J\longmapsto(ra_1,\ldots,ra_n),
\]
is an injective homomorphism. Conversely, every homomorphism from $R/J$
to $R^n$ has this form, by taking the image of $1+J$.
The case $x=0$ causes no exception.

For every ideal $K$, one has $K\subseteq\ann_R(\ann_R K)$ and
$\ann_R K=\ann_R(\ann_R(\ann_R K))$.
Consequently every annihilator ideal is fixed by double annihilation.
If $R$ is Noetherian and $J=\ann_R(\ann_R J)$, the ideal $\ann_R J$
is finitely generated and supplies the required representation of $J$.
\end{proof}

The finite-generation issue in the annihilator semilattice is controlled
by an order condition on the faithful quotient. An \emph{annihilator
ideal} of a ring $R$ means an ideal of the form $\ann_R(I)$, where $I$
is an ideal of $R$. The ascending chain condition on annihilator ideals
is weaker than Noetherianity.

\begin{theorem}\label{thm:finite-detection-acc}
Let $E\ne0$ be an annihilator multiplication $A$-module. If
$R=A/\ann_A(E)$ satisfies the ascending chain condition on annihilator
ideals, then there is a finitely generated submodule $F\subseteq E$
such that $\ann_A(F)=\ann_A(E)$. In particular, this conclusion holds
when the faithful quotient $R$ is Noetherian.
\end{theorem}

\begin{proof}
Work over $R$. Annihilation is an order-reversing involution on the
set of annihilator ideals. Thus the ascending chain condition on that
set also gives the descending chain condition. By
Lemma~\ref{lem:fg-submodule} and faithfulness, each member of
\[
 \{\ann_R(F):F\subseteq E\text{ is finitely generated}\}
\]
is an annihilator ideal. Choose a minimal member $\ann_R(F)$.
For every $x\in E$, the ideal
$\ann_R(F+Rx)=\ann_R(F)\cap\ann_R(x)$ lies in the same collection
and is contained in $\ann_R(F)$. Minimality gives
$\ann_R(F)\subseteq\ann_R(x)$ for all $x$. Hence
$\ann_R(F)=\ann_R(E)=0$. Taking inverse images in $A$ proves the claim.
\end{proof}

The next lemma isolates consequences of finite annihilator detection
that do not themselves require annihilator multiplication.

\begin{lemma}\label{lem:finite-detection-calculus}
Suppose that an $A$-module $E$ contains a finitely generated submodule
$F$ with $\ann_A(F)=\ann_A(E)=B$. For every finitely generated ideal
$I$ and every multiplicatively closed subset $S$ of $A$, one has
\[
 \begin{aligned}
 \ann_A(IF)&=\ann_A(IE)=(B:I),\\
 \Supp_A(IE)&=\Supp_A(IF)=V((B:I)),\\
 \ann_{S^{-1}A}(S^{-1}E)&=S^{-1}B.
 \end{aligned}
\]
\end{lemma}

\begin{proof}
The first equality follows from \eqref{eq:colon-annihilator}.
Since $IF$ is finitely generated, its support is $V((B:I))$.
The inclusion $IF\subseteq IE$ gives one inclusion of supports, while
$\Supp_A(IE)\subseteq V(\ann_A(IE))=V((B:I))$ gives the other.
Localization of an annihilator commutes with localization for a finitely
generated module: a common denominator can be chosen for its finitely
many generators. Thus
$\ann_{S^{-1}A}(S^{-1}F)=S^{-1}B$.
The inclusions $F\subseteq E$ and $BE=0$ now give both inclusions
in the last equality.
\end{proof}

\begin{theorem}\label{thm:localization-detection}
Let $E$ be an annihilator multiplication $A$-module containing a finitely
generated submodule $F$ with $\ann_A(F)=\ann_A(E)$.
Then $S^{-1}E$ is an annihilator multiplication $S^{-1}A$-module
for every multiplicatively closed subset $S$ of $A$.
In particular, this holds without finite generation of $E$ whenever
$A/\ann_A(E)$ satisfies the ascending chain condition on annihilator ideals.
\end{theorem}

\begin{proof}
Put $B=\ann_A(E)$. For $x\in E$, choose a finitely generated ideal
$I$ with $\ann_A(x)=(B:I)$.
Annihilators of cyclic modules, and colons by finitely generated ideals,
commute with localization. Together with
Lemma~\ref{lem:finite-detection-calculus}, this gives
\[
 \ann_{S^{-1}A}(x/1)=S^{-1}(B:I)
 =(S^{-1}B:S^{-1}I)
 =\ann_{S^{-1}A}\bigl((S^{-1}I)(S^{-1}E)\bigr).
\]
Multiplication by the unit $1/s$ does not change the annihilator of $x/1$.
The last assertion follows from Theorem~\ref{thm:finite-detection-acc}.
\end{proof}

The preceding theorem extends the finitely generated localization result
of \cite[Proposition~3]{KocAMM}. Over a Noetherian faithful quotient,
finite detection is also exactly the additional requirement in a
local-to-global statement.

\begin{theorem}\label{thm:noetherian-local-global}
Let $E\ne0$ be an $A$-module and suppose that
$R=A/\ann_A(E)$ is Noetherian. Then $E$ is annihilator multiplication
if and only if both of the following conditions hold:
\begin{enumerate}[label=\textup{(\roman*)}]
\item some finitely generated $F\subseteq E$ satisfies
$\ann_A(F)=\ann_A(E)$;
\item $E_{\mathfrak m}$ is annihilator multiplication over
$R_{\mathfrak m}$ for every $\mathfrak m\in\Max(R)$.
\end{enumerate}
In particular, for finitely generated $E$, the property is local at the
maximal ideals of its faithful quotient.
\end{theorem}

\begin{proof}
Necessity follows from Theorems~\ref{thm:finite-detection-acc} and
\ref{thm:localization-detection}.
For sufficiency, work over $R$ and let $J=\ann_R(x)$ for $x\in E$.
Condition (i) and Lemma~\ref{lem:finite-detection-calculus} show that
$E_{\mathfrak m}$ is faithful over $R_{\mathfrak m}$.
All ideals of $R$, including $J$ and $\ann_R J$, are finitely generated.
Applying Proposition~\ref{prop:cyclic-embedding} locally and commuting
the two annihilators with localization gives
\[
 J_{\mathfrak m}
 =\ann_{R_{\mathfrak m}}\bigl(\ann_{R_{\mathfrak m}}
         (J_{\mathfrak m})\bigr)
 =\bigl(\ann_R(\ann_R J)\bigr)_{\mathfrak m}.
\]
Equality at every maximal ideal yields $J=\ann_R(\ann_R J)$.
Proposition~\ref{prop:cyclic-embedding} finishes the proof.
\end{proof}

\begin{example}\label{ex:local-global-obstruction}
Condition (i) cannot be omitted, even for $A=\mathbb Z$.
The faithful module $E=\bigoplus_p\mathbb Z/p\mathbb Z$ has
$E_{(q)}\cong\mathbb Z/q\mathbb Z$ for every prime number $q$,
so every maximal localization is cyclic and annihilator multiplication.
Globally, however, $\ann_{\mathbb Z}(IE)$ is either $0$ or $\mathbb Z$
for every ideal $I$, whereas $p\mathbb Z$ occurs as an element
annihilator. Thus $E$ is not annihilator multiplication.
Each finitely generated submodule has finite prime support and a nonzero
annihilator, so no submodule as in (i) exists.
\end{example}

\subsection{Faithful modules and universal module conditions}

Recall that a commutative ring is \emph{quasi-Frobenius} if it is
Artinian and self-injective. Equivalently, it is Artinian and every
ideal $J$ satisfies $J=\ann(\ann J)$; this is a classical
characterization, not a new assertion here; see \cite{Storrer69,Lam99}.
The annihilator multiplication formulation below does not require the
ring to be reduced.

\begin{theorem}\label{thm:faithful-qf}
For a nonzero commutative ring $R$, the following conditions are equivalent:
\begin{enumerate}[label=\textup{(\roman*)}]
\item every faithful $R$-module is annihilator multiplication;
\item every faithful $R$-module generated by at most two elements is
annihilator multiplication;
\item $R\oplus R/J$ is annihilator multiplication for every ideal $J$;
\item every ideal of $R$ is the annihilator of a finitely generated ideal;
\item $R$ is quasi-Frobenius.
\end{enumerate}
\end{theorem}

\begin{proof}
The first condition implies the second, which implies the third.
The module in (iii) is faithful. Applying its defining condition to
$(0,1+J)$ proves (iv).

Suppose (iv) holds. Every ideal is an annihilator ideal and therefore
is fixed by double annihilation. Given any ideal $J$, apply (iv) to
$\ann_R J$ to obtain a finitely generated ideal $K$ such that
$\ann_R J=\ann_R K$. Double annihilation gives $J=K$.
Hence $R$ is Noetherian. Since annihilation is now an order-reversing
involution on the entire ideal lattice, the ascending chain condition
also gives the descending chain condition. Thus $R$ is Artinian,
and the classical double-annihilator characterization gives (v).

If (v) holds, then $\ann_R J$ is finitely generated and
$J=\ann_R(\ann_R J)$ for every ideal $J$, proving (iv).
Finally, assume (iv), let $E$ be faithful, and write
$\ann_R(x)=\ann_R I$ with $I$ finitely generated.
Faithfulness yields $\ann_R I=\ann_R(IE)$, proving (i).
\end{proof}

\begin{remark}
For reduced rings, Theorem~\ref{thm:faithful-qf} recovers the finite
product of fields characterization in \cite[Theorem~5]{KocAMM}, because
reduced commutative Artinian rings are finite products of fields.
The additional content is the removal of reducedness and the
identification of the resulting class as the quasi-Frobenius rings.
\end{remark}

An \emph{Artinian principal ideal ring} is a commutative Artinian ring
in which every ideal is principal. Equivalently, it is a finite product
of Artinian local rings whose ideals are the successive powers of a
principal maximal ideal. Requiring the property for all modules,
rather than only the faithful ones, gives this smaller class.

\begin{theorem}\label{thm:all-modules-pir}
For a nonzero commutative ring $R$, the following conditions are equivalent:
\begin{enumerate}[label=\textup{(\roman*)}]
\item every $R$-module is annihilator multiplication;
\item every $R$-module generated by at most two elements is annihilator
multiplication;
\item $R/B\oplus R/J$ is annihilator multiplication whenever $B\subseteq J$
are ideals of $R$;
\item $R/B$ is quasi-Frobenius for every proper ideal $B$ of $R$;
\item $R$ is an Artinian principal ideal ring.
\end{enumerate}
\end{theorem}

\begin{proof}
The first two implications are immediate. For fixed $B$, the modules
in (iii) are precisely the tests in Theorem~\ref{thm:faithful-qf}(iii)
over $R/B$. Lemma~\ref{lem:quotient-change} therefore proves (iv).

Assume (iv). Taking $B=0$ shows that $R$ is Artinian.
Decompose it into its finitely many Artinian local factors, and let
$(S,\mathfrak m)$ be one such factor. The ring $S/\mathfrak m^2$ is
a quotient of $R$, hence is quasi-Frobenius.
Put $\mathfrak n=\mathfrak m/\mathfrak m^2$.
In a local ring with square-zero maximal ideal $\mathfrak n$, the
annihilator of every nonzero proper ideal is $\mathfrak n$.
If $\dim_{S/\mathfrak m}\mathfrak n\ge2$, a one-dimensional subspace
$L\subsetneq\mathfrak n$ is an ideal with
$\ann(\ann L)=\ann\mathfrak n=\mathfrak n\ne L$, a contradiction.
Thus $\mathfrak m/\mathfrak m^2$ has dimension at most one.
Nakayama's lemma shows that $\mathfrak m$ is principal (or zero).
If $\mathfrak m=(\pi)$ and $\pi^n=0$, every nonzero element is
$u\pi^j$ for a unit $u$ and some $0\le j<n$: choose its largest
power of $\mathfrak m$ containing it. An ideal is consequently generated
by an element with the least such exponent among its nonzero elements.
Every ideal of $S$ is a power of $(\pi)$, proving (v).

Conversely, it suffices to prove (i) over an Artinian local principal
ideal ring $S$, since modules over a finite product decompose by the
coordinate idempotents, and the defining annihilator equalities hold
coordinatewise. The field case is immediate. Otherwise write
$\mathfrak m=(\pi)$, with $\pi^n=0$ and $\pi^{n-1}\ne0$.
For a nonzero module $E$, put $\ann_S(E)=(\pi^b)$.
For $0\ne x\in E$, write $\ann_S(x)=(\pi^a)$, where
$1\le a\le b\le n$. The ideal $I=(\pi^{b-a})$ satisfies
\[
 \ann_S(IE)=((\pi^b):(\pi^{b-a}))=(\pi^a)=\ann_S(x).
\]
The colon equality follows by comparing the exponents of nonzero
principal ideals. For $x=0$ take $I=0$; the zero module is also allowed.
This proves (i).
\end{proof}

The equivalence between the classical ring conditions in (iv) and (v)
is also consistent with the quotient self-injectivity results discussed
in \cite{Storrer69}. Here quotients include $R$ itself; excluding the
zero ideal would give a different ring-theoretic question.

\begin{corollary}\label{cor:closure-pir}
For a commutative ring $R$, each of the following properties is equivalent
to $R$ being an Artinian principal ideal ring: the class of annihilator
multiplication modules is closed under finite direct sums; it is closed
under factor modules; it is closed under extensions.
\end{corollary}

\begin{proof}
Every cyclic module is annihilator multiplication: in $R/B$, take
$I=(a)$ for the element $a+B$.
Closure under finite direct sums therefore gives
Theorem~\ref{thm:all-modules-pir}(iii).
Closure under factor modules gives its condition (ii), since every
two-generated module is a quotient of the annihilator multiplication
module $R^2$. Closure under extensions implies closure under finite
direct sums by considering split exact sequences.
Conversely, over an Artinian principal ideal ring every module has the
property, so all three closure assertions hold. The zero ring is harmless.
\end{proof}

\begin{example}\label{ex:qf-not-universal}
The ring $R=k[x,y]/(x^2,y^2)$ is quasi-Frobenius, but not every
$R$-module is annihilator multiplication. Indeed, the functional taking
the coefficient of $xy$ gives a nondegenerate multiplication pairing
$R\times R\to k$. For an ideal $I$, its orthogonal complement under
this pairing is $\ann_R I$: if $aI\ne0$, nondegeneracy and the ideal
property supply $b\in R$ and $i\in I$ with a nonzero $xy$ coefficient
in $abi$. Thus double orthogonality proves the double-annihilator
property for all ideals, and $R$ is Artinian.

Consider $E=R/(xy)\oplus R/(x)$, whose annihilator is $(xy)$.
Over its faithful quotient $S=R/(xy)=k[x,y]/(x,y)^2$, the element
$(0,1+(x))$ has annihilator $(x)$.
The annihilators of ideals of $S$ are only $0$, $(x,y)$, and $S$.
Hence $E$ is not annihilator multiplication, even though its two cyclic
summands are. This gives an explicit distinction between
Theorems~\ref{thm:faithful-qf} and \ref{thm:all-modules-pir}.
\end{example}

\begin{corollary}\label{cor:abelian-groups}
An abelian group is an annihilator multiplication $\mathbb Z$-module
if and only if it is torsion-free or has bounded exponent.
\end{corollary}

\begin{proof}
The zero group is immediate, so assume $E\ne0$.
A faithful module over a domain is annihilator multiplication exactly
when it is torsion-free: annihilators of nonzero ideals of a domain
are zero, and one takes $I=R$ for a nonzero torsion-free element.
This also recovers \cite[Theorem~4]{KocAMM}.
If $\ann_{\mathbb Z}(E)=n\mathbb Z$ with $n>0$, work over the
Artinian principal ideal ring $\mathbb Z/n\mathbb Z$ and apply
Theorem~\ref{thm:all-modules-pir}.
Conversely, a nonfaithful abelian group has bounded exponent.
\end{proof}

\subsection{Reduced rings with finitely many minimal primes}

For rings with zero divisors, call a module \emph{regular-torsion-free}
if multiplication by every non-zero-divisor of the ring is injective.
This is weaker than requiring every nonzero element of the module to
have zero annihilator. Let $Q(R)$ denote the total quotient ring,
obtained by inverting the non-zero-divisors of $R$.

\begin{theorem}\label{thm:reduced-finite-min}
Let $R$ be a nonzero reduced ring with finitely many minimal primes,
and let $E$ be a faithful $R$-module. Then the following conditions
are equivalent:
\begin{enumerate}[label=\textup{(\roman*)}]
\item $E$ is annihilator multiplication;
\item $E$ is regular-torsion-free;
\item the canonical map $E\to Q(R)\otimes_R E$ is injective.
\end{enumerate}
When these conditions hold, every $x\in E$ has
$\ann_R(x)=\ann_R(aE)$ for a single element $a\in R$.
\end{theorem}

\begin{proof}
If $E$ is annihilator multiplication and $cx=0$ with $c$ a
non-zero-divisor, write $\ann_R(x)=\ann_R I$ with $I$ finitely
generated. Then $cI=0$, so $I=0$ and $x=0$.
This proves (i) implies (ii), and the standard kernel description of
localization gives the equivalence of (ii) and (iii).

Write the minimal primes as $\mathfrak p_1,\ldots,\mathfrak p_s$.
For each $i$, choose
$d_i\in\bigcap_{j\ne i}\mathfrak p_j\setminus\mathfrak p_i$;
products of elements in $\mathfrak p_j\setminus\mathfrak p_i$
provide such a choice, and take $d_1=1$ when $s=1$.
Reducedness gives $d_id_j=0$ for $i\ne j$.
An element of $R$ is a non-zero-divisor exactly when it lies outside
all the $\mathfrak p_i$: one direction follows from the injection
$R\to\prod_i R/\mathfrak p_i$, and an element in $\mathfrak p_i$
is killed by the nonzero $d_i$.
Thus $d=\sum_i d_i$ is a non-zero-divisor.

For clarity, the familiar decomposition of the total quotient ring is
\[
 Q(R)\cong\prod_{i=1}^s\operatorname{Frac}(R/\mathfrak p_i).
\]
Indeed, the map into this product is injective. A fraction $a/b$ in
its $i$th coordinate, with $b\notin\mathfrak p_i$, and zero in the
other coordinates, is represented by
\[
 \frac{a d_i}{b d_i+\sum_{j\ne i}d_j}.
\]
Its denominator is outside every minimal prime, proving surjectivity.
The coordinate idempotents are $e_i=d_i/d$.

Assume (iii), identify $E$ with its image in $Q(R)\otimes_R E$, and
put $S_x=\{i:e_ix\ne0\}$.
Each coordinate is a vector space over $\operatorname{Frac}(R/\mathfrak p_i)$,
so
\[
 \ann_R(x)=\bigcap_{i\in S_x}\mathfrak p_i
 =\ann_R\left(\sum_{i\in S_x}d_i\right).
\]
Faithfulness changes the last expression into
$\ann_R((\sum_{i\in S_x}d_i)E)$.
An empty sum is zero and an empty intersection is $R$, covering $x=0$.
\end{proof}

\begin{example}\label{ex:infinite-min-obstruction}
The finiteness of the minimal-prime set cannot be omitted in general.
Let $R=\prod_{i\ge1}k$, let $J=\bigoplus_{i\ge1}k\varepsilon_i$,
and put $E=R\oplus R/J$.
Every non-zero-divisor of $R$ is a unit, so this faithful module is
regular-torsion-free. However, $\ann_R(0,1+J)=J$, while
$\ann_R J=0$ and $\ann_R(\ann_R J)=R\ne J$.
Proposition~\ref{prop:cyclic-embedding} rules out annihilator multiplication.
The ring is reduced and has infinitely many minimal primes, since its
infinitely many coordinate kernels are distinct maximal ideals in a
von Neumann regular ring and hence are minimal.
\end{example}

\section{Support, projective constructions, and associated primes}
\label{sec:support}

The annihilator multiplication condition describes the support of each
finitely generated submodule using ideal multiples. Finite annihilator
detection, rather than finite generation of the ambient module, is the
hypothesis needed for the support equalities. The projective, trace, and
associated-prime results impose only the finiteness assumptions stated
explicitly below.

\subsection{Support decompositions}

\begin{proposition}
\label{prop:support-decomposition}
Let $E$ be an annihilator multiplication $A$-module containing a
finitely generated submodule $F$ with $\ann_A(F)=\ann_A(E)$. If $L$
is a finitely generated submodule of $E$, then there exists a finitely generated
ideal $I=(a_1,\ldots,a_n)$ of $A$ such that
\[
\Supp_A(L)=\Supp_A(IE)=\bigcup_{i=1}^n \Supp_A(a_iE).
\]
\end{proposition}

\begin{proof}
By Lemma~\ref{lem:fg-submodule}, there is a finitely generated ideal $I$ with
$\ann_A(L)=\ann_A(IE)$. Write $I=(a_1,\ldots,a_n)$. The module $L$ is finitely generated, and
Lemma~\ref{lem:finite-detection-calculus} gives the corresponding support
formula for $IE$. Therefore
\[
\Supp_A(L)=V(\ann_A(L))=V(\ann_A(IE))=\Supp_A(IE).
\]
Moreover $IE=a_1E+\cdots+a_nE$. After localization at a prime ideal
$\mathfrak p$, the module $(IE)_{\mathfrak p}$ is nonzero exactly when
$(a_iE)_{\mathfrak p}$ is nonzero for at least one $i$. Hence
\[
\Supp_A(IE)=\bigcup_{i=1}^n \Supp_A(a_iE),
\]
as required.
\end{proof}

When the support of $L$ is irreducible, the preceding finite union has only one
essential term.

\begin{corollary}
\label{cor:irreducible-support}
Let $E$ be an annihilator multiplication $A$-module containing a
finitely generated submodule $F$ with $\ann_A(F)=\ann_A(E)$, and let
$L$ be a finitely generated submodule of $E$. Suppose that $\Supp_A(L)$ is
irreducible. If $I=(a_1,\ldots,a_n)$ is a finitely generated ideal satisfying
$\ann_A(L)=\ann_A(IE),$
then there exists $j\in\{1,\ldots,n\}$ such that $\Supp_A(L)=\Supp_A(a_jE).$
In particular,
$\sqrt{\ann_A(L)}=\sqrt{\ann_A(a_jE)}.$
\end{corollary}

\begin{proof}
Proposition~\ref{prop:support-decomposition} gives
\[
\Supp_A(L)=\bigcup_{i=1}^n \Supp_A(a_iE).
\]
Lemma~\ref{lem:finite-detection-calculus} shows that every
$\Supp_A(a_iE)=V(\ann_A(a_iE))$ is closed in $\Spec(A)$. An irreducible closed set cannot be a finite union of proper closed
subsets. Hence $\Supp_A(L)=\Supp_A(a_jE)$ for some $j$. The annihilator descriptions of these supports give equality of the
displayed radicals.
\end{proof}

The same decomposition gives a useful test for containment in open subsets of
$\Spec(A)$.

\begin{corollary}
\label{cor:local-support-test}
Let $E$ be an annihilator multiplication $A$-module containing a
finitely generated submodule $F$ with $\ann_A(F)=\ann_A(E)$, and let
$L$ be a finitely generated submodule of $E$. Let $U\subseteq\Spec(A)$ be open.
Then the following conditions are equivalent:
\begin{enumerate}[label=\textup{(\roman*)}]
  \item $\Supp_A(L)\subseteq U$;
  \item for some finitely generated ideal $I=(a_1,\ldots,a_n)$ satisfying
  $\ann_A(L)=\ann_A(IE)$, one has
  \[
    \Supp_A(a_iE)\subseteq U\qquad\text{for every }i=1,\ldots,n.
  \]
\end{enumerate}
\end{corollary}

\begin{proof}
This is immediate from the equality
\[
\Supp_A(L)=\bigcup_{i=1}^n \Supp_A(a_iE)
\]
in Proposition~\ref{prop:support-decomposition}.
\end{proof}

\subsection{Projective modules and tensor products}

We now turn to projective modules. For a projective $A$-module $P$
and $x\in P$, define its \emph{order ideal} by
$\mathcal O_P(x)=\{\varphi(x):\varphi\in\Hom_A(P,A)\}$.
Even when $P$ is not finitely generated, its dual basis has finite
support at each fixed element.

\begin{theorem}\label{thm:fg-projective}
Every projective $A$-module is annihilator multiplication.
More precisely, for every $x\in P$, the ideal $\mathcal O_P(x)$ is
finitely generated and
$\ann_A(x)=\ann_A(\mathcal O_P(x)P)$.
\end{theorem}

\begin{proof}
A projective module is a direct summand of a free module, possibly of
infinite rank. Restricting coordinate maps and projecting basis vectors
gives a dual basis $(x_\lambda,\varphi_\lambda)_{\lambda\in\Lambda}$
such that, for every $y\in P$, only finitely many
$\varphi_\lambda(y)$ are nonzero and
$y=\sum_\lambda\varphi_\lambda(y)x_\lambda$; see also \cite{WK24}.
Fix $x$ and let $\Lambda_x$ be its finite support. Then
\[
 \mathcal O_P(x)=(\varphi_\lambda(x):\lambda\in\Lambda_x),
\]
since applying any functional to the dual-basis expression for $x$
gives one inclusion, and the other is immediate.
If $ax=0$, then $a\varphi(x)=\varphi(ax)=0$ for every functional,
so $a\mathcal O_P(x)P=0$. Conversely, $x\in\mathcal O_P(x)P$
by its dual-basis expression, so an element annihilating this ideal
multiple also annihilates $x$.
\end{proof}

\begin{corollary}
Every free module, with no restriction on its rank, is annihilator
multiplication.
\end{corollary}

\begin{proof}
Every free module is projective.
\end{proof}

The exact tensor-product criterion will follow from the trace theorem
below; in particular, faithfulness of the projective factor will not be
needed for preservation.

\subsection{Hom modules and traces}

For $A$-modules $M,N$, the \emph{trace of $M$ in $N$} is
\[
\Tr_M(N)=\sum_{f\in\Hom_A(M,N)}f(M).
\]
An element of $A$ annihilates all maps $M\to N$ precisely when it annihilates
their images. Thus, without any finiteness hypotheses,
\begin{equation}\label{eq:hom-trace-annihilator}
\ann_A\bigl(\Hom_A(M,N)\bigr)=\ann_A\bigl(\Tr_M(N)\bigr).
\end{equation}
This equality and \eqref{eq:colon-annihilator} connect their annihilator
multiplication properties.

\begin{proposition}\label{prop:trace-hom}\label{prophom}
Let $M$ be a finitely generated $A$-module and let $N$ be an $A$-module.
If $\Tr_M(N)$ is annihilator multiplication, then $\Hom_A(M,N)$ is
annihilator multiplication.
\end{proposition}

\begin{proof}
Put $T=\Tr_M(N)$ and $H=\Hom_A(M,N)$. For $f\in H$, its image is a
finitely generated submodule of $T$. Lemma~\ref{lem:fg-submodule} supplies
a finitely generated ideal $I$ with
$\ann_A(f)=\ann_A(f(M))=\ann_A(IT)$.
Equations~\eqref{eq:colon-annihilator} and
\eqref{eq:hom-trace-annihilator} give $\ann_A(IT)=\ann_A(IH)$,
which is the required representation.
\end{proof}

We recall the trace-ideal description for finitely generated projective
modules; see also \cite{Lam99}. A proof is included to make the subsequent
argument independent of reference numbering.

\begin{lemma}\label{lem:projective-trace}
Let $P$ be a finitely generated projective $A$-module. There is an idempotent
$e\in A$ such that $\Tr_P(A)=eA$, $P=eP$, and
$\Tr_P(N)=eN$ for every $A$-module $N$.
\end{lemma}

\begin{proof}
The case $P=0$ is immediate with $e=0$. Otherwise choose a finite dual basis
$p_1,\ldots,p_r$ and $\lambda_1,\ldots,\lambda_r\in\Hom_A(P,A)$, so
$p=\sum_i\lambda_i(p)p_i$ for every $p\in P$.
Put $T=\Tr_P(A)$. The dual-basis identities show that
$T=(\lambda_i(p_j):1\leq i,j\leq r)$ and $P=TP$.
Moreover, for $\alpha\in\Hom_A(P,A)$ and $p\in P$,
$\alpha(p)=\sum_i\lambda_i(p)\alpha(p_i)\in T^2$.
Hence $T=T^2$.

For completeness, a finitely generated ideal $T=T^2$ is generated by an
idempotent. Write $T=(t_1,\ldots,t_q)$ and
$t_i=\sum_j b_{ij}t_j$ with $b_{ij}\in T$.
The adjugate identity shows that $d=\det(\delta_{ij}-b_{ij})$ annihilates
$T$. Since $d\equiv1\pmod T$, write $d=1-e$ with $e\in T$.
Then $(1-e)T=0$, so $e^2=e$ and $T=eA$.
Consequently $P=eP$.

It follows that every map $P\to N$ has image in $eN$.
Conversely, write $e=\sum_{j=1}^s\alpha_j(z_j)$ with
$\alpha_j\in\Hom_A(P,A)$ and $z_j\in P$.
For $n\in N$, the maps $p\mapsto\alpha_j(p)n$ show that
$en=\sum_j\alpha_j(z_j)n$ belongs to $\Tr_P(N)$.
Thus $\Tr_P(N)=eN$.
\end{proof}

Projectivity supplies the reverse implication in
Proposition~\ref{prop:trace-hom}.

\begin{theorem}\label{thm:hom-trace-projective}
Let $P$ be a finitely generated projective $A$-module and let $N$ be an
$A$-module. Then $\Hom_A(P,N)$ is annihilator multiplication if and only if
$\Tr_P(N)$ is annihilator multiplication.
\end{theorem}

\begin{proof}
One implication is Proposition~\ref{prop:trace-hom}. Conversely, suppose
$H=\Hom_A(P,N)$ is annihilator multiplication. By
Lemma~\ref{lem:projective-trace}, write $T=\Tr_P(N)=eN$ and
$e=\sum_{j=1}^s\alpha_j(z_j)$ as in its proof.
For $n\in T$, define $g_j\in H$ by $g_j(p)=\alpha_j(p)n$ and set
$K=Ag_1+\cdots+Ag_s$. Each $g_j(P)$ is contained in $An$, while
$n=en=\sum_j g_j(z_j)$. Therefore
\[
\sum_{j=1}^s g_j(P)=An,
\qquad
\ann_A(K)=\ann_A(n).
\]
Since $K$ is finitely generated, Lemma~\ref{lem:fg-submodule} gives a
finitely generated ideal $I$ with $\ann_A(K)=\ann_A(IH)$.
Using \eqref{eq:hom-trace-annihilator} and \eqref{eq:colon-annihilator},
we obtain $\ann_A(n)=\ann_A(IH)=\ann_A(IT)$.
The zero-trace case is immediate. Hence $T$ is annihilator multiplication.
\end{proof}

The dual of a finitely generated projective module has the same trace
ideal. This yields both preservation and reflection for tensor products.

\begin{theorem}\label{thm:tensor-trace}\label{prop:tensor-faithful-projective}
Let $P$ be a finitely generated projective $A$-module, and write
$\Tr_P(A)=eA$ with $e^2=e$. For every $A$-module $N$,
\[
 P\otimes_A N\text{ is annihilator multiplication}
 \quad\iff\quad eN\text{ is annihilator multiplication}.
\]
In particular, tensoring by $P$ preserves annihilator multiplication.
If $P$ is faithful, it also reflects the property.
\end{theorem}

\begin{proof}
Put $P^*=\Hom_A(P,A)$. The bidual identification $P\cong P^{**}$
shows that the values defining $\Tr_{P^*}(A)$ are the same values
$\alpha(p)$ defining $\Tr_P(A)$.
A finite dual basis gives the natural isomorphism
$P\otimes_A N\cong\Hom_A(P^*,N)$, under which
$p\otimes n$ corresponds to $\alpha\mapsto\alpha(p)n$.
Apply Theorem~\ref{thm:hom-trace-projective} to $P^*$ and use
Lemma~\ref{lem:projective-trace} to identify the trace with $eN$.

If $N$ is annihilator multiplication and $x\in eN$, choose a finitely
generated ideal $I$ with $\ann_A(x)=\ann_A(IN)$.
Since $(1-e)x=0$, one has $(1-e)IN=0$, so
$IN=eIN=(eI)(eN)$. This proves that $eN$ is annihilator
multiplication, and hence proves preservation.
Finally, if $P$ is faithful, $P=eP$ implies $1-e\in\ann_A(P)=0$,
so $e=1$, giving reflection.
\end{proof}

\subsection{Associated primes}

For an $A$-module $M$, the set $\Ass_A(M)$ consists of the prime ideals
$\ann_A(x)$ with $0\ne x\in M$; write $\Ass(A)=\Ass_A(A)$.
No Noetherian hypothesis is imposed. In
\cite[Theorem~6(ii)]{KocAMM}, equality of $\Ass_A(E)$ and $\Ass(A)$ is
proved when $E$ is annihilator multiplication and contains an element with
zero annihilator. The following result answers
\cite[Question~1]{KocAMM} by replacing that last condition with faithfulness.

\begin{theorem}\label{thm:associated-primes}
Let $E$ be a faithful annihilator multiplication $A$-module. Then
\[
\Ass_A(E)=\Ass(A).
\]
\end{theorem}

\begin{proof}
We include the argument for both inclusions. First let
$\mathfrak p=\ann_A(x)\in\Ass_A(E)$.
Choose a finitely generated ideal $I=(b_1,\ldots,b_r)$ with
$\ann_A(x)=\ann_A(IE)$. Faithfulness gives
\[
\mathfrak p=\ann_A(I)=\bigcap_{i=1}^r\ann_A(b_i).
\]
The product of the ideals $\ann_A(b_i)$ is contained in $\mathfrak p$.
Since $\mathfrak p$ is prime, some $\ann_A(b_j)$ is contained in
$\mathfrak p$, and the displayed intersection gives the reverse inclusion.
Thus $\mathfrak p=\ann_A(b_j)$ and $\mathfrak p\in\Ass(A)$.
This is the inclusion in \cite[Theorem~6(i)]{KocAMM}.

Conversely, let $\mathfrak p=\ann_A(a)\in\Ass(A)$ with $a\ne0$.
Faithfulness gives $aE\ne0$, so choose $x\in E$ with $ax\ne0$.
There is a finitely generated ideal $I$ with
$\ann_A(x)=\ann_A(IE)=\ann_A(I)$.
If $I\subseteq\mathfrak p$, then $aI=0$, giving $ax=0$, a contradiction.
Hence $I\not\subseteq\mathfrak p$.
Clearly $\mathfrak p\subseteq\ann_A(ax)$.
If $cax=0$, then $ca\in\ann_A(x)=\ann_A(I)$, and consequently
$cI\subseteq\ann_A(a)=\mathfrak p$.
Primality and $I\not\subseteq\mathfrak p$ imply $c\in\mathfrak p$.
Thus $\ann_A(ax)=\mathfrak p$, proving $\mathfrak p\in\Ass_A(E)$.
\end{proof}

Passing to the faithful quotient gives a version for arbitrary annihilator
multiplication modules.

\begin{corollary}\label{cor:associated-faithful-quotient}
Let $E$ be an annihilator multiplication $A$-module and set
$B=\ann_A(E)$. Then
\[
\Ass_A(E)=\Ass_A(A/B).
\]
\end{corollary}

\begin{proof}
The zero-module case is immediate. Otherwise,
Lemma~\ref{lem:quotient-change} makes $E$ a faithful annihilator
multiplication $A/B$-module. Apply Theorem~\ref{thm:associated-primes}
over $A/B$ and take inverse images of the associated primes under
$A\to A/B$. For any $A/B$-module, element annihilators over $A$ are exactly
these inverse images.
\end{proof}

\section{Annihilator classes, graphs, and torsion}
\label{sec:graph}

This section studies three related descriptions of annihilator data:
a support-disjointness graph, a quotient semilattice, and the part of a
Gabriel filter detected by finitely generated submodules. These descriptions
record different information; in particular, graph adjacency does not recover
all annihilator ideals.

\subsection{The support-disjointness graph}

\begin{definition}
Let $E$ be a nonzero $A$-module. The \emph{support-disjointness graph} of $E$, denoted
by $\GammaSD(E)$, is the graph whose vertices are the annihilator classes
\[
[L]=\{K\subseteq E:K\text{ is finitely generated and } \ann_A(K)=\ann_A(L)\},
\]
where $L$ ranges over the nonzero finitely generated submodules of $E$. Two
distinct vertices $[L]$ and $[K]$ are adjacent when
\[
\Supp_A(L)\cap\Supp_A(K)=\varnothing.
\]
Equivalently,
\[
[L]\sim[K] \quad\iff\quad \ann_A(L)+\ann_A(K)=A.
\]
\end{definition}

The graph is simple and undirected. A vertex is \emph{isolated} if it has
no neighbours, and \emph{universal} if it is adjacent to every other vertex.
A graph is \emph{edgeless} if it has no edges and \emph{complete} if every
two distinct vertices are adjacent. A graph is \emph{connected} if any two
vertices can be joined by a finite path. The complete graph on one vertex
is denoted by $K_1$ and is connected by convention.

For annihilator multiplication modules, adjacency can be checked after replacing
submodules by finitely generated ideal multiples of the ambient module.

\begin{proposition}
\label{prop:graph-test}
Let $E$ be a nonzero annihilator multiplication $A$-module. Let $L$
and $K$ be nonzero finitely generated submodules of $E$. Choose finitely
generated ideals $I$ and $J$ of $A$ such that
\[
\ann_A(L)=\ann_A(IE),\qquad \ann_A(K)=\ann_A(JE).
\]
Then the corresponding vertices $[L]$ and $[K]$ of $\GammaSD(E)$ are adjacent
if and only if
\[
\ann_A(IE)+\ann_A(JE)=A.
\]
\end{proposition}

\begin{proof}
Since $L$ and $K$ are finitely generated,
\[
\Supp_A(L)=V(\ann_A(L)),\qquad \Supp_A(K)=V(\ann_A(K)).
\]
Thus $\Supp_A(L)\cap\Supp_A(K)=\varnothing$ if and only if
\[
V(\ann_A(L)+\ann_A(K))=\varnothing,
\]
which is equivalent to $\ann_A(L)+\ann_A(K)=A$. Substituting the chosen
annihilator representations gives the result.
\end{proof}

\begin{theorem}
\label{thm:graph-edgeless}
Let $E$ be a nonzero $A$-module. Suppose
that the set
\[
\{\ann_A(L):L\subseteq E\text{ is a nonzero finitely generated submodule}\}
\]
is totally ordered by inclusion. Then $\GammaSD(E)$ has no edges.
\end{theorem}

\begin{proof}
Let $[L]$ and $[K]$ be distinct vertices. Since $L$ and $K$ are nonzero, both
$\ann_A(L)$ and $\ann_A(K)$ are proper ideals. By the total-order hypothesis,
one of these annihilators is contained in the other. Hence their sum is equal
to one of them and is still proper. The definition of $\GammaSD(E)$ then
shows that $[L]$ and $[K]$ are not adjacent.
\end{proof}

\begin{corollary}
Under the hypotheses of Theorem~\ref{thm:graph-edgeless}, the complementary
support-intersection graph, with the same vertices and with edges corresponding
to nonempty intersections of supports, is complete.
\end{corollary}

\begin{proof}
The support-intersection graph is the complement of $\GammaSD(E)$ on the same
vertex set.
\end{proof}

\subsection{Prime modules and isolated vertices}

A nonzero $A$-module $E$ is a \emph{prime module} if
$ax=0$ with $a\in A$ and $0\ne x\in E$ implies $aE=0$.
Equivalently, $\ann_A(K)=\ann_A(E)$ for every nonzero submodule $K$ of $E$.
This is the condition that the zero submodule be prime.

\begin{theorem}\label{thm:graph-prime}\label{thm1}
Let $E$ be a nonzero $A$-module.
\begin{enumerate}[label=\textup{(\roman*)}]
\item $\GammaSD(E)=K_1$ if and only if $E$ is a prime module.
\item If $E$ is finitely generated, then $\GammaSD(E)$ is connected if and
only if $E$ is a prime module. In that case $\GammaSD(E)=K_1$.
\end{enumerate}
\end{theorem}

\begin{proof}
If $\GammaSD(E)=K_1$, then all nonzero cyclic submodules have the same
annihilator. Its intersection over all nonzero elements of $E$ is
$\ann_A(E)$, so every nonzero element has annihilator $\ann_A(E)$.
The same is then true for every nonzero submodule, proving that $E$ is prime.
The converse follows directly from the definition.

For the second assertion, suppose $E$ is finitely generated.
The vertex $[E]$ is isolated: for any nonzero finitely generated
$K\subseteq E$, one has
$\ann_A(E)+\ann_A(K)=\ann_A(K)\ne A$.
Thus connectedness forces $\GammaSD(E)=K_1$, and the first assertion applies.
Conversely, a prime module gives $K_1$, which is connected.
\end{proof}

\begin{remark}\label{rem:finite-annihilator-detection}
The proof of Theorem~\ref{thm:graph-prime}(ii) only requires a finitely
generated submodule $F\subseteq E$ with $\ann_A(F)=\ann_A(E)$:
the vertex $[F]$ is then isolated. This weaker condition also appears in
Proposition~\ref{prop:lattice-quotient}.
\end{remark}

The finite-generation hypothesis cannot be omitted from the connectedness
statement, even though the vertices themselves are defined using finitely
generated submodules.

\begin{example}\label{ex:connected-infinite-torsion}
Let $\mathcal P$ be the set of prime numbers and consider the $\mathbb Z$-module
$E=\bigoplus_{p\in\mathcal P}\mathbb Z/p\mathbb Z$.
Then $\ann_{\mathbb Z}(E)=0$, whereas the summand $\mathbb Z/2\mathbb Z$
has annihilator $2\mathbb Z$, so $E$ is not prime.
A finitely generated submodule has nonzero coordinates at only finitely many
primes. If $S$ is that finite nonempty set, its annihilator is
$(\prod_{p\in S}p)\mathbb Z$.
Hence the vertices are indexed by the finite nonempty subsets of $\mathcal P$,
with adjacency corresponding to disjointness.
Given two such sets $S,T$, choose a prime $q\notin S\cup T$.
The vertex indexed by $\{q\}$ is adjacent to both. The graph is therefore
connected. Its diameter, the supremum of shortest-path lengths, is exactly
two: distinct intersecting sets give nonadjacent vertices.
This example concerns arbitrary modules; it is not an annihilator
multiplication module, since faithfulness would force
$\ann_{\mathbb Z}(IE)=\ann_{\mathbb Z}(I)$, which is either $0$ or
$\mathbb Z$, whereas $2\mathbb Z$ occurs as an element annihilator.
Example~\ref{ex:semilattice-gap} gives the same connectedness phenomenon
within faithful annihilator multiplication modules.
\end{example}

The existence of sums of finitely generated submodules imposes a strong
restriction on universal vertices.

\begin{proposition}\label{prop:graph-universal}
Let $E$ be a nonzero $A$-module.
\begin{enumerate}[label=\textup{(\roman*)}]
\item If $\GammaSD(E)$ has a universal vertex, then $\GammaSD(E)=K_1$.
\item $\GammaSD(E)$ is complete if and only if $\GammaSD(E)=K_1$, and
these conditions hold if and only if $E$ is a prime module.
\end{enumerate}
\end{proposition}

\begin{proof}
Let $[L]$ be universal and suppose $[N]\ne[L]$ is another vertex.
Since $L+N$ is finitely generated and
$\ann_A(L+N)+\ann_A(L)=\ann_A(L)\ne A$,
universality forces $[L+N]=[L]$.
Thus $\ann_A(L)\subseteq\ann_A(N)$.
But universality also gives
$A=\ann_A(L)+\ann_A(N)=\ann_A(N)$, contradicting $N\ne0$.
This proves the first assertion. Every vertex of a nonempty complete graph
is universal, and the second assertion follows from
Theorem~\ref{thm:graph-prime}(i).
\end{proof}

For a proper ideal $I$ of $A$, write
$J(I)=\bigcap_{\mathfrak m\in\Max(A),\ I\subseteq\mathfrak m}\mathfrak m$.
Thus $J(I)/I$ is the Jacobson radical of $A/I$.
For a nonzero module $E$, also put
\[
\Sigma_E=\Supp_A(E)\cap\Max(A),\qquad
\mathfrak j_E=\bigcap_{\mathfrak m\in\Sigma_E}\mathfrak m.
\]
The set $\Sigma_E$ is nonempty: a maximal ideal containing the annihilator
of any nonzero cyclic submodule belongs to it.

\begin{proposition}\label{prop:graph-isolated}\label{pisolated}
Let $E$ be a nonzero $A$-module and let $L\subseteq E$ be a nonzero finitely
generated submodule.
\begin{enumerate}[label=\textup{(\roman*)}]
\item If $\ann_A(L)\subseteq\mathfrak j_E$, then $[L]$ is isolated in
$\GammaSD(E)$. In particular, this holds if
$J(\ann_A(L))=J(\ann_A(E))$.
\item Suppose $\Sigma_E\subseteq\Ass_A(E)$. Then $[L]$ is isolated if and
only if $\ann_A(L)\subseteq\mathfrak j_E$.
If, in addition, $E$ is finitely generated, these conditions are equivalent
to $J(\ann_A(L))=J(\ann_A(E))$.
\end{enumerate}
\end{proposition}

\begin{proof}
Suppose $\ann_A(L)\subseteq\mathfrak j_E$ and let $[K]$ be a vertex.
Choose a maximal ideal $\mathfrak m$ containing $\ann_A(K)$.
Finite generation gives $K_{\mathfrak m}\ne0$, hence $E_{\mathfrak m}\ne0$.
Thus $\mathfrak m\in\Sigma_E$ and
$\ann_A(L)+\ann_A(K)\subseteq\mathfrak m\ne A$.
This proves isolation. Since
$\Sigma_E\subseteq V(\ann_A(E))\cap\Max(A)$, one has
$J(\ann_A(E))\subseteq\mathfrak j_E$.
Therefore $J(\ann_A(L))=J(\ann_A(E))$ implies
$\ann_A(L)\subseteq\mathfrak j_E$ without assuming that $E$ is finitely
generated.

Now assume $\Sigma_E\subseteq\Ass_A(E)$ and $[L]$ is isolated.
If $\ann_A(L)\not\subseteq\mathfrak m$ for some $\mathfrak m\in\Sigma_E$,
choose $0\ne x\in E$ with $\ann_A(x)=\mathfrak m$.
Then $\ann_A(L)+\ann_A(x)=A$, so $[L]$ and $[Ax]$ are distinct adjacent
vertices, a contradiction. This proves the converse.

Finally, if $E$ is finitely generated, then
$\Sigma_E=V(\ann_A(E))\cap\Max(A)$ and
$\mathfrak j_E=J(\ann_A(E))$.
Since $\ann_A(E)\subseteq\ann_A(L)$, the inclusion
$\ann_A(L)\subseteq J(\ann_A(E))$ holds precisely when these two
annihilators are contained in the same maximal ideals, equivalently when
their $J$-ideals are equal.
\end{proof}

The associated-prime hypothesis in the converse cannot be omitted in general.

\begin{example}\label{ex:isolated-hypothesis}
Let $A=k[X,Y]$, where $k$ is algebraically closed, and put $E=A/(XY)$.
The module is cyclic and hence annihilator multiplication: for
$x=\overline f$, the ideal $I=(f)$ satisfies $IE=Ax$.
Since $(XY)=(X)\cap(Y)$, the annihilator of each nonzero element of $E$
is one of $(XY),(X),(Y)$. These three ideals all occur, and their finite
intersections are again among them. Thus they are exactly the annihilators
of nonzero finitely generated submodules, and
$\Ass_A(E)=\{(X),(Y)\}$.
All three annihilators are contained in the maximal ideal $(X,Y)$, so
$\GammaSD(E)$ is edgeless.

Take $L=A\overline X$, for which $\ann_A(L)=(Y)$.
The maximal ideal $(X,Y-1)$ belongs to $\Sigma_E$ but does not contain $(Y)$.
Consequently $[L]$ is isolated although
$\ann_A(L)\not\subseteq\mathfrak j_E$.
More precisely, intersecting the point maximal ideals on the two coordinate
axes gives $\mathfrak j_E=(XY)=J(\ann_A(E))$:
a polynomial vanishing at all points of both axes is divisible by both
$X$ and $Y$. Since $(X)$ and $(Y)$ are not maximal,
$\Sigma_E\not\subseteq\Ass_A(E)$, as required.
\end{example}

Applying the isolated-vertex criterion to every vertex gives an edgelessness
test.

\begin{corollary}\label{cor:graph-edgeless-support}
Let $E$ be a nonzero $A$-module with $\Sigma_E\subseteq\Ass_A(E)$.
Then $\GammaSD(E)$ is edgeless if and only if
$\ann_A(L)\subseteq\mathfrak j_E$ for every nonzero finitely generated
submodule $L\subseteq E$.
If $E$ is finitely generated, this is equivalent to
$J(\ann_A(L))=J(\ann_A(E))$ for every such $L$.
\end{corollary}

\begin{proof}
A graph is edgeless precisely when all its vertices are isolated.
Apply Proposition~\ref{prop:graph-isolated}.
\end{proof}

\subsection{The annihilator semilattice}

Let $\FGId(A)$ be the set of finitely generated ideals of $A$, including
$0$. Define
\[
\begin{aligned}
\mathcal A_E^{\mathrm{fg}}
 &=\{\ann_A(L):L\subseteq E\text{ is finitely generated}\},\\
\mathcal B_E^{\mathrm{fg}}
 &=\{\ann_A(IE):I\in\FGId(A)\}.
\end{aligned}
\]
The zero submodule is included in $\mathcal A_E^{\mathrm{fg}}$.
The two collections must be distinguished because $IE$ need not be finitely
generated when $I$ is finitely generated. A \emph{semilattice} here means a
set with an associative, commutative, idempotent binary operation; the
operations under consideration are ideal sum and ideal intersection.

\begin{proposition}\label{prop:lattice-quotient}
Let $E$ be an annihilator multiplication $A$-module and define
$I\sim_E J$ on $\FGId(A)$ by $\ann_A(IE)=\ann_A(JE)$.
Then ideal sum induces a well-defined operation on $\FGId(A)/{\sim_E}$,
and the map
\[
\Phi:\FGId(A)/{\sim_E}\longrightarrow\mathcal B_E^{\mathrm{fg}},
\qquad [I]\longmapsto\ann_A(IE),
\]
is a bijection satisfying
$\Phi([I+J])=\Phi([I])\cap\Phi([J])$.
Moreover, $\mathcal A_E^{\mathrm{fg}}\subseteq\mathcal B_E^{\mathrm{fg}}$,
and the following conditions are equivalent:
\begin{enumerate}[label=\textup{(\roman*)}]
\item $\mathcal A_E^{\mathrm{fg}}=\mathcal B_E^{\mathrm{fg}}$;
\item there exists a finitely generated submodule $F\subseteq E$ such that
$\ann_A(F)=\ann_A(E)$.
\end{enumerate}
In particular, if $E$ is finitely generated, $\Phi$ is a bijection onto
$\mathcal A_E^{\mathrm{fg}}$.
\end{proposition}

\begin{proof}
The identity
\[
\ann_A((I+J)E)=\ann_A(IE)\cap\ann_A(JE)
\]
shows that $\sim_E$ is compatible with ideal sum, and proves the asserted
compatibility of $\Phi$ with the operations.
The map is well-defined, injective, and surjective onto
$\mathcal B_E^{\mathrm{fg}}$ by construction.
Lemma~\ref{lem:fg-submodule} gives
$\mathcal A_E^{\mathrm{fg}}\subseteq\mathcal B_E^{\mathrm{fg}}$.

If the two collections coincide, take $I=A$ to obtain a finitely generated
$F\subseteq E$ with $\ann_A(F)=\ann_A(E)$.
Conversely, suppose such $F$ exists. For each finitely generated ideal $I$,
the submodule $IF\subseteq E$ is finitely generated and
\[
\ann_A(IF)=(\ann_A(F):I)=(\ann_A(E):I)=\ann_A(IE).
\]
Hence $\mathcal B_E^{\mathrm{fg}}\subseteq\mathcal A_E^{\mathrm{fg}}$.
\end{proof}

The following example shows why an unconditional bijection onto
$\mathcal A_E^{\mathrm{fg}}$ is not available.

\begin{example}\label{ex:semilattice-gap}
Let $k$ be a field, put $A=\prod_{i\geq1}k$, and let $\varepsilon_i\in A$
be the $i$th coordinate idempotent. Consider the ideal
$E=\bigoplus_{i\geq1}k\varepsilon_i$ as an $A$-module.
For a finite set $S$ of positive integers, write
$e_S=\sum_{i\in S}\varepsilon_i$.
If $x\in E$ has nonzero coordinates exactly at $S$, then
$Ax=e_SE$. Thus $E$ is annihilator multiplication, with the principal ideal
$I=Ae_S$ representing $\ann_A(x)$.
The module is faithful, since an element annihilating every
$\varepsilon_i$ has all coordinates zero.

Every finitely generated submodule $L\subseteq E$ has its coordinates in a
finite set $S$, and, taking $S$ to be the union of the nonzero coordinates
of generators, one has
$\ann_A(L)=(1-e_S)A\ne0$.
Consequently $\ann_A(E)=0$ lies in $\mathcal B_E^{\mathrm{fg}}$ but not
in $\mathcal A_E^{\mathrm{fg}}$.
In particular, the proposed value $\Phi([A])=0$ cannot belong to the latter
collection.

The nonzero annihilator classes are indexed by finite nonempty sets $S$,
and two vertices are adjacent precisely when their index sets are disjoint.
A singleton outside the union of two given finite sets supplies a path of
length two. Thus $\GammaSD(E)$ is connected and has diameter two, while
$E$ is not prime. This also shows that the finiteness assumption in
Theorem~\ref{thm:graph-prime}(ii) cannot be dropped even for faithful
annihilator multiplication modules.
\end{example}

\subsection{Graph rigidity and finite minimal-prime formulas}

Finite annihilator detection makes the graph independent of the particular
faithful annihilator multiplication module. Only its faithful quotient
ring remains.

\begin{theorem}\label{thm:graph-rigidity}
Let $E$ be a nonzero annihilator multiplication $A$-module.
Put $R=A/\ann_A(E)$ and suppose that its annihilator ideals satisfy
the ascending chain condition. Then their identification gives a
canonical graph isomorphism
\[
 \GammaSD(E)\cong\GammaSD(R).
\]
Here the graph on the right is computed for the regular $R$-module $R$.
Consequently $\GammaSD(E)$ is connected if and only if $R$ is a domain;
in this case the graph has one vertex.
\end{theorem}

\begin{proof}
Pass to $R$. Theorem~\ref{thm:finite-detection-acc} supplies a finitely
generated faithful submodule $F\subseteq E$.
Proposition~\ref{prop:lattice-quotient} gives
\[
 \{\ann_R(L):L\subseteq E\text{ is finitely generated}\}
 =\{\ann_R(I):I\subseteq R\text{ is finitely generated}\}.
\]
A member equals $R$ exactly when its represented submodule is zero.
After deleting this member, these are the two vertex sets, and in each
case adjacency is comaximality of annihilators.
Taking inverse images under $A\to R$ preserves comaximality, so this
also proves the assertion over $A$.
The regular module $R$ is finitely generated, and it is a prime module
exactly when $R$ is a domain. Apply Theorem~\ref{thm:graph-prime}.
\end{proof}

We first recall the graph-theoretic terminology used below.
Let $G$ be a finite simple graph with at least one vertex, and denote
its vertex set and edge set by $V(G)$ and $E(G)$, respectively.
The \emph{degree} of a vertex $v\in V(G)$, denoted by $\deg_G(v)$,
is the number of vertices adjacent to $v$; we write $\deg(v)$ when
the graph is clear from the context. A vertex of degree zero is
called \emph{isolated}.
A \emph{clique} in $G$ is a nonempty set of vertices any two distinct
members of which are adjacent. The \emph{clique number} $\omega(G)$
is the maximum cardinality of a clique in $G$.
A \emph{proper vertex coloring} of $G$ is an assignment of colors
to its vertices such that adjacent vertices receive different colors.
The \emph{chromatic number} $\chi(G)$ is the minimum number of colors
needed for a proper vertex coloring of $G$.
Finally, $|E(G)|$ denotes the number of edges of $G$.

For a reduced ring with finitely many minimal primes, the next theorem
determines the support-disjointness graph and these invariants from
the pairwise comaximality of the minimal primes.

\begin{theorem}\label{thm:finite-min-graph}
Let $R$ be reduced with distinct minimal primes
$\mathfrak p_1,\ldots,\mathfrak p_s$, and let $E$ be a faithful
annihilator multiplication $R$-module. Then $\GammaSD(E)$ has exactly
$2^s-1$ vertices, indexed by the nonempty subsets $S\subseteq\{1,\ldots,s\}$
through the annihilator $\bigcap_{i\in S}\mathfrak p_i$.
Distinct $S,T$ are adjacent exactly when
\[
 \mathfrak p_i+\mathfrak p_j=R
 \quad\text{for every }i\in S\text{ and }j\in T.
\]
Writing
$C(S)=\{j:\mathfrak p_i+\mathfrak p_j=R\text{ for every }i\in S\}$,
one has $\deg(S)=2^{|C(S)|}-1$.
Let $H$ be the graph on $\{1,\ldots,s\}$ whose edges are the pairs
of comaximal minimal primes. Then
\[
 \omega(\GammaSD(E))=\omega(H),\qquad
 \chi(\GammaSD(E))=\chi(H).
\]
\end{theorem}

\begin{proof}
Choose $d_i$ as in Theorem~\ref{thm:reduced-finite-min}.
For every ideal $I$ of $R$, reducedness gives
\[
 \ann_R I=\bigcap_{\{i:I\not\subseteq\mathfrak p_i\}}\mathfrak p_i.
\]
Indeed, a product $aI$ is zero exactly when it is zero modulo every
minimal prime, and $R/\mathfrak p_i$ is a domain.
There are therefore finitely many annihilator ideals, so
Theorems~\ref{thm:finite-detection-acc} and \ref{thm:graph-rigidity}
apply. Every subset $S$ is realized by the principal ideal generated
by $\sum_{i\in S}d_i$.
Different subsets give different intersections: $d_i$ is outside the
intersection exactly when $i$ belongs to the indexing subset.
The empty subset gives $R$ and corresponds to the zero submodule;
all the other subsets give vertices.

The closed set of $\bigcap_{i\in S}\mathfrak p_i$ is
$\bigcup_{i\in S}V(\mathfrak p_i)$.
Two such sets are disjoint exactly when all cross pairs of minimal
primes are comaximal, proving the adjacency criterion.
The neighbors of $S$ are precisely the nonempty subsets of $C(S)$.
Since $S\cap C(S)=\varnothing$, none of them equals $S$, and the
degree formula follows.

The singleton subsets induce a copy of $H$.
Conversely, choose one index $i(S)\in S$ for every nonempty $S$.
The map $S\mapsto i(S)$ sends adjacent vertices to adjacent vertices
of $H$. Pulling back a coloring of $H$ gives a coloring of
$\GammaSD(E)$, and the map is injective on every clique.
Together with the induced copy of $H$, this proves both equalities.
\end{proof}

\begin{corollary}\label{cor:comaximal-min-graph}
Under the hypotheses of Theorem~\ref{thm:finite-min-graph}, suppose
that the minimal primes are pairwise comaximal. Then adjacency is
set-disjointness, and
\[
 \deg(S)=2^{s-|S|}-1,\qquad
 |E(\GammaSD(E))|=\frac{3^s-2^{s+1}+1}{2},\qquad
 \omega(\GammaSD(E))=\chi(\GammaSD(E))=s.
\]
The vertex corresponding to the full set is isolated.
\end{corollary}

\begin{proof}
Here $C(S)$ is the complement of $S$ and $H$ is complete.
For the edge count, assign each index to $S$, to $T$, or to neither;
there are $3^s$ ordered disjoint pairs, allowing empty sets.
Deleting pairs with $S=\varnothing$ or $T=\varnothing$ leaves
$3^s-2^{s+1}+1$ ordered pairs of nonempty disjoint sets.
Each edge is counted twice.
\end{proof}

\subsection{Hereditary torsion theories}

A torsion theory $\tau=(\mathcal T,\mathcal F)$ on
$A\operatorname{-Mod}$ consists of classes with
$\Hom_A(T,F)=0$ for $T\in\mathcal T$ and $F\in\mathcal F$, such that every
module $M$ has an exact sequence $0\to T\to M\to F\to0$ with
$T\in\mathcal T$ and $F\in\mathcal F$.
It is \emph{hereditary} if $\mathcal T$ is closed under submodules.
Modules in $\mathcal T$ are called $\tau$-torsion. Fix such a theory and write
\[
\mathfrak F_\tau=\{J\subseteq A:A/J\in\mathcal T\}
\]
for its Gabriel filter, a collection of ideals closed under upward inclusion
and finite intersections. These closure properties also follow directly:
$A/K$ is a quotient of $A/J$ when $J\subseteq K$, and
$A/(J_1\cap J_2)$ embeds in $A/J_1\oplus A/J_2$.
The $\tau$-torsion submodule of an $A$-module $M$ is
\[
t_\tau(M)=\{x\in M:\ann_A(x)\in\mathfrak F_\tau\}.
\]

\begin{theorem}
\label{thm:torsion-test}
Let $E$ be an annihilator multiplication $A$-module, and let $L$ be a finitely
generated submodule of $E$. Choose a finitely generated ideal $I$ of $A$ such
that $\ann_A(L)=\ann_A(IE).$
Then the following conditions are equivalent:
\begin{enumerate}[label=\textup{(\roman*)}]
  \item $L$ is $\tau$-torsion;
  \item $\ann_A(L)\in\mathfrak F_\tau$;
  \item $\ann_A(IE)\in\mathfrak F_\tau$.
\end{enumerate}
If, in addition, some finitely generated $F\subseteq E$ satisfies
$\ann_A(F)=\ann_A(E)$, then these conditions are also equivalent to
$IE\in\mathcal T$.
\end{theorem}

\begin{proof}
Write $L=Ax_1+\cdots+Ax_n$. Then
\[
\ann_A(L)=\bigcap_{i=1}^n\ann_A(x_i).
\]
The module $L$ is $\tau$-torsion if and only if each generator $x_i$ is
$\tau$-torsion, equivalently if $\ann_A(x_i)\in\mathfrak F_\tau$ for all $i$.
Since the Gabriel filter is closed under finite intersections and upward
inclusion, this is equivalent to
\[
\ann_A(L)=\bigcap_{i=1}^n\ann_A(x_i)\in\mathfrak F_\tau.
\]
Thus (i) and (ii) are equivalent. The equality
$\ann_A(L)=\ann_A(IE)$ gives the equivalence of (ii) and (iii).

For the final assertion, $IF$ is finitely generated and
$\ann_A(IF)=\ann_A(IE)$ by Lemma~\ref{lem:finite-detection-calculus}.
If $IE$ is $\tau$-torsion, its submodule $IF$ is $\tau$-torsion,
so $\ann_A(IE)=\ann_A(IF)\in\mathfrak F_\tau$.
Conversely, if $\ann_A(IE)$ belongs to the filter, then so does the
annihilator of every element of $IE$, by upward closure.
Hence $IE$ is $\tau$-torsion.
\end{proof}

\begin{corollary}
\label{cor:elementwise-torsion}
Let $E$ be an annihilator multiplication $A$-module. For $e\in E$, choose a
finitely generated ideal $I$ of $A$ such that $\ann_A(e)=\ann_A(IE).$
Then
\[
e\in t_\tau(E) \quad\iff\quad \ann_A(IE)\in\mathfrak F_\tau.
\]
If some finitely generated $F\subseteq E$ has
$\ann_A(F)=\ann_A(E)$, then
\[
e\in t_\tau(E) \quad\iff\quad IE\in\mathcal T.
\]
\end{corollary}

\begin{proof}
Apply Theorem~\ref{thm:torsion-test} to the cyclic submodule $Ae$.
\end{proof}

Specializing to the hereditary torsion theory determined by a multiplicatively
closed subset gives the usual localization criterion.

\begin{corollary}
\label{cor:s-torsion}
Let $S\subseteq A$ be a multiplicatively closed subset, and let
\[
T_S(E)=\{e\in E:se=0\text{ for some }s\in S\}
\]
be the usual $S$-torsion submodule of $E$. Let $L$ be a finitely generated
submodule of an annihilator multiplication $A$-module $E$, and choose a finitely
generated ideal $I$ of $A$ such that $\ann_A(L)=\ann_A(IE)$. Then the following
conditions are equivalent:
\begin{enumerate}[label=\textup{(\roman*)}]
  \item $L$ is $S$-torsion;
  \item $S^{-1}L=0$;
  \item $\ann_A(L)\cap S\neq\varnothing$;
  \item $\ann_A(IE)\cap S\neq\varnothing$.
\end{enumerate}
If, in addition, some finitely generated $F\subseteq E$ satisfies
$\ann_A(F)=\ann_A(E)$, these conditions are also equivalent to
$S^{-1}(IE)=0$.
\end{corollary}

\begin{proof}
The hereditary torsion theory associated to $S$ has Gabriel filter
\[
    \mathfrak F_S=\{J\subseteq A:J\cap S\neq\varnothing\}.
\]
For finitely generated $L$, one has
\[
L\text{ is }S\text{-torsion} \quad\iff\quad S^{-1}L=0 \quad\iff\quad \ann_A(L)\cap S\neq\varnothing.
\]
Using $\ann_A(L)=\ann_A(IE)$ gives the equivalence with (iv).
Under the additional hypothesis, $IF$ is finitely generated and has
the same annihilator as $IE$. If $S^{-1}(IE)=0$, then
$S^{-1}(IF)=0$, so this common annihilator meets $S$.
Conversely, an element in $\ann_A(IE)\cap S$ kills $IE$ and becomes
a unit upon localization. This proves the final assertion.
\end{proof}

The torsion finitely generated submodules are therefore exactly the annihilator
classes lying in the Gabriel filter.

\begin{proposition}
\label{prop:torsion-annihilator-classes}
Let $E$ be an annihilator multiplication $A$-module. For a hereditary
torsion theory $\tau$, set
\[
\mathcal A_{E,\tau}^{\operatorname{fg}}
=\{\ann_A(L):L\subseteq E\text{ is finitely generated and }L\in\mathcal T\}.
\]
Then
$\mathcal A_{E,\tau}^{\operatorname{fg}}=\mathcal A_E^{\operatorname{fg}}\cap\mathfrak F_\tau.$
Moreover, $\mathcal A_{E,\tau}^{\operatorname{fg}}$ is closed under finite
intersections and is upward closed inside $\mathcal A_E^{\operatorname{fg}}$.
\end{proposition}

\begin{proof}
By Theorem~\ref{thm:torsion-test}, a finitely generated submodule $L$ is
$\tau$-torsion if and only if $\ann_A(L)\in\mathfrak F_\tau$. This proves the
displayed equality.

Since $\mathfrak F_\tau$ is upward closed, the same is true for
$\mathcal A_{E,\tau}^{\operatorname{fg}}$ inside
$\mathcal A_E^{\operatorname{fg}}$. If
$J_1=\ann_A(L_1)$ and $J_2=\ann_A(L_2)$ belong to
$\mathcal A_{E,\tau}^{\operatorname{fg}}$, then
\[
J_1\cap J_2=\ann_A(L_1)\cap\ann_A(L_2)=\ann_A(L_1+L_2).
\]
The submodule $L_1+L_2$ is finitely generated, and
$J_1\cap J_2\in\mathfrak F_\tau$ because Gabriel filters are closed under finite
intersections. Therefore
$J_1\cap J_2\in\mathcal A_{E,\tau}^{\operatorname{fg}}$.
\end{proof}

\begin{remark}
By Theorem~\ref{thm:finite-detection-acc}, the additional hypotheses
in these torsion statements hold for every annihilator multiplication
module whose faithful quotient satisfies the ascending chain condition
on annihilator ideals, in particular when that quotient is Noetherian.
\end{remark}

\section{Amalgamations and square-zero constructions}
\label{sec:amalgamation}

We first treat annihilator multiplication modules over amalgamated rings,
and then give structural square-zero results and explicit nonprojective
constructions. Let $f:A\to B$ be a ring homomorphism, let $J$ be an ideal of $B$,
let $M$ be an $A$-module, and let $N$ be a $B$-module.
Regard $N$ as an $A$-module through $f$, and let
$\varphi:M\to N$ be an $A$-module homomorphism.
Following El Khalfaoui, Mahdou, Sahandi, and Shirmohammadi
\cite{ElKhalfaouiEtAl2021}, the \emph{amalgamation of $M$ and $N$
along $J$ with respect to $\varphi$} is
\[
M\bowtie^\varphi JN
=
\{(m,\varphi(m)+n):m\in M,\ n\in JN\},
\]
viewed as an $A\bowtie^fJ$-module under coordinatewise addition
and scalar multiplication.
The coordinatewise action makes it an $R$-module: if $c=f(a)+j$, then
$c(\varphi(m)+n)-\varphi(am)=j\varphi(m)+cn\in JN$.
Put
\[
C=f(A)+J, \qquad X=\varphi(M)+JN.
\]
Then $C$ is a subring of $B$ and $X$ is a $C$-module, since
$(f(a)+j)\varphi(m)=\varphi(am)+j\varphi(m)\in X$ and $C(JN)\subseteq JN$.
There are natural
surjective homomorphisms
\[
\rho_A:R\longrightarrow A, \qquad \rho_A(a,f(a)+j)=a,
\]
and
\[
\rho_C:R\longrightarrow C, \qquad \rho_C(a,f(a)+j)=f(a)+j.
\]

\subsection{Coordinate annihilators}

The following computation expresses annihilators in the amalgamated module as
intersections of annihilators coming from the two coordinates.

\begin{lemma}
\label{lem:amalgamated-annihilator}
Let
\[
x=(m,\varphi(m)+n)\in M\bowtie^\varphi JN, \qquad n\in JN.
\]
Then
\[
\ann_R(x) =\rho_A^{-1}\bigl(\ann_A(m)\bigr) \cap \rho_C^{-1}\bigl(\ann_C(\varphi(m)+n)\bigr).
\]
Moreover, if $H$ is an ideal of $R$, then
\[
\ann_R(HE) =\rho_A^{-1}\bigl(\ann_A(\rho_A(H)M)\bigr) \cap \rho_C^{-1}\bigl(\ann_C(\rho_C(H)X)\bigr).
\]
\end{lemma}

\begin{proof}
Let $r=(a,f(a)+j)\in R$. Then $rx=0$ if and only if both coordinates vanish,
that is,
\[
am=0 \qquad\text{and}\qquad (f(a)+j)(\varphi(m)+n)=0.
\]
This is equivalent to
\[
a\in\ann_A(m) \qquad\text{and}\qquad f(a)+j\in\ann_C(\varphi(m)+n),
\]
which proves the formula for $\ann_R(x)$.

For the second formula, the first projection of $HE$ is $\rho_A(H)M$, and the
second projection is $\rho_C(H)X$. Hence $r\in R$ annihilates $HE$ if and only
if $\rho_A(r)$ annihilates $\rho_A(H)M$ and $\rho_C(r)$ annihilates
$\rho_C(H)X$. This gives the displayed equality.
\end{proof}

\begin{theorem}
\label{thm:amalgamated-characterization}
The amalgamated module $E=M\bowtie^\varphi JN$ is an annihilator multiplication
$R$-module if and only if, for every $m\in M$ and every $n\in JN$, there exists
a finitely generated ideal $H$ of $R$ such that
\[
\begin{aligned}
&\rho_A^{-1}\bigl(\ann_A(m)\bigr) \cap \rho_C^{-1}\bigl(\ann_C(\varphi(m)+n)\bigr) \\
&\qquad =\rho_A^{-1}\bigl(\ann_A(\rho_A(H)M)\bigr) \cap \rho_C^{-1}\bigl(\ann_C(\rho_C(H)X)\bigr).
\end{aligned}
\]
\end{theorem}

\begin{proof}
By definition, $E$ is an annihilator multiplication $R$-module if and only if,
for every $x=(m,\varphi(m)+n)\in E$, there exists a finitely generated ideal
$H$ of $R$ such that $\ann_R(x)=\ann_R(HE)$. Lemma~\ref{lem:amalgamated-annihilator}
translates this equality into the displayed condition.
\end{proof}

Equivalently, for each $x=(m,\varphi(m)+n)\in E$, the criterion asks for
finitely many $h_i=(a_i,c_i)\in R$, where $c_i=f(a_i)+j_i$, such that
\[
\ann_R(x)=
\left\{(a,c)\in R:\begin{array}{l}
 aa_iM=0\text{ and }cc_iX=0\\
 \text{for every }i=1,\ldots,t
\end{array}\right\}.
\]
Here $H=(h_1,\ldots,h_t)$. The two projected ideals are obtained from this
single ideal $H$; the criterion does not assert that they may be chosen
independently. It is an exact coordinate reformulation of the defining
condition, rather than a componentwise transfer theorem in full generality.

\subsection{Product and trivial-amalgamation cases}

The characterization recovers the expected componentwise statement in the
product case (\cite[Proposition 2]{KocAMM}).

\begin{corollary}
\label{cor:product-case}
Assume that $J=B$. Then $A\bowtie^fJ=A\times B$ and
$M\bowtie^\varphi JN=M\times N$. In this case, $M\bowtie^\varphi JN$ is an
annihilator multiplication $(A\times B)$-module if and only if $M$ is an
annihilator multiplication $A$-module and $N$ is an annihilator multiplication
$B$-module.
\end{corollary}

\begin{proof}
If $J=B$, then every element of $A\times B$ has the form $(a,f(a)+j)$ for some
$a\in A$ and $j\in B$. Thus $A\bowtie^fJ=A\times B$. Since $JN=N$, one also
has $M\bowtie^\varphi JN=M\times N$. For $(m,n)\in M\times N$,
\[
\ann_{A\times B}(m,n)=\ann_A(m)\times\ann_B(n).
\]
Every ideal of $A\times B$ is $I\times K$, by multiplication with the
coordinate idempotents, and it is finitely generated precisely when both
$I$ and $K$ are finitely generated. For such ideals,
\[
\ann_{A\times B}((I\times K)(M\times N))=\ann_A(IM)\times\ann_B(KN).
\]
If $M$ and $N$ have the property, choose representing ideals $I$ and $K$
for $m$ and $n$; their product gives the required representation.
Conversely, apply the property of $M\times N$ to $(m,0)$ and compare first
coordinates, and then to $(0,n)$ and compare second coordinates.
\end{proof}

The opposite extreme is the case where the amalgamated part acts trivially on
the second coordinate.

\begin{corollary}
\label{cor:trivial-amalgamated-part}
Assume that $JN=0$. Then
\[
M\bowtie^\varphi JN=\{(m,\varphi(m)):m\in M\}.
\]
Moreover, $M\bowtie^\varphi JN$ is an annihilator multiplication
$A\bowtie^fJ$-module if and only if $M$ is an annihilator multiplication
$A$-module.
\end{corollary}

\begin{proof}
Since $JN=0$, we have $J\varphi(M)\subseteq JN=0$. Therefore
\[
M\bowtie^\varphi JN=\{(m,\varphi(m)):m\in M\}.
\]
The map
\[
\theta:M\bowtie^\varphi JN\longrightarrow M, \qquad \theta(m,\varphi(m))=m,
\]
is an isomorphism of modules when $M$ is viewed as an
$A\bowtie^fJ$-module through the projection
$\rho_A:A\bowtie^fJ\to A$. Indeed,
\[
(a,f(a)+j)(m,\varphi(m))=(am,(f(a)+j)\varphi(m))=(am,\varphi(am)),
\]
because $j\varphi(m)=0$. The kernel of $\rho_A$ annihilates
$M\bowtie^\varphi JN$, so Lemma~\ref{lem:quotient-change} shows that the
annihilator multiplication property over $A\bowtie^fJ$ is equivalent to the same
property over $A$.
\end{proof}

\begin{example}
Let $k$ be a field, let $C=k[\varepsilon]/(\varepsilon^2)$ and write again $\varepsilon$ for the residue class of $\varepsilon$ in $C$. Set $J=(\varepsilon)$. Put $A=k$ and let
$\alpha:A\longrightarrow C$ be the natural inclusion sending $a$ to its residue class as a constant
polynomial. Then
\[
R=A\bowtie^\alpha J=\{(a,a+\lambda\varepsilon):a,\lambda\in k\}.
\]
Thus $R$ is naturally isomorphic to $C$, but it is written as an
amalgamated algebra along the nonzero square-zero ideal $J$.

Let $M=A=k$, let $N=C$, and define
\[
\varphi:M\longrightarrow N,\qquad \varphi(a)=a.
\]
Then $JN=J=(\varepsilon)$, and the corresponding amalgamated module is
\[
E=M\bowtie^\varphi JN = \{(a,a+\lambda\varepsilon):a,\lambda\in k\}.
\]
Hence $E=R$ as an $R$-module. In particular, $E$ is a cyclic multiplication
$R$-module, and hence it is an annihilator multiplication module.

We now verify directly the criterion in
Theorem~\ref{thm:amalgamated-characterization}. Let $m\in M=k$ and let
$n=\lambda\varepsilon\in JN$. Then
\[
    \varphi(m)+n=m+\lambda\varepsilon\in C.
\]
We must find a finitely generated ideal $H$ of $R$ such that
\[
\begin{aligned}
&\rho_A^{-1}\bigl(\ann_A(m)\bigr)
 \cap
 \rho_C^{-1}\bigl(\ann_C(m+\lambda\varepsilon)\bigr) \\
&\qquad =
\rho_A^{-1}\bigl(\ann_A(\rho_A(H)M)\bigr)
\cap
\rho_C^{-1}\bigl(\ann_C(\rho_C(H)N)\bigr).
\end{aligned}
\]
Take the principal ideal $H=R(m,m+\lambda\varepsilon).$
Then $\rho_A(H)$ is the ideal of $A=k$ generated by $m$, and $\rho_C(H)$ is
the ideal of $C$ generated by $m+\lambda\varepsilon$. Therefore
\[
    \rho_A(H)M=mA
    \quad\text{and}\quad
    \rho_C(H)N=C(m+\lambda\varepsilon).
\]
Hence
\[
    \ann_A(\rho_A(H)M)=\ann_A(m)
\]
and
\[
    \ann_C(\rho_C(H)N)
    =
    \ann_C\bigl(C(m+\lambda\varepsilon)\bigr)
    =
    \ann_C(m+\lambda\varepsilon).
\]
Substituting these two equalities gives exactly
\[
\begin{aligned}
&\rho_A^{-1}\bigl(\ann_A(m)\bigr)
 \cap
 \rho_C^{-1}\bigl(\ann_C(m+\lambda\varepsilon)\bigr) \\
&\qquad =
\rho_A^{-1}\bigl(\ann_A(\rho_A(H)M)\bigr)
\cap
\rho_C^{-1}\bigl(\ann_C(\rho_C(H)N)\bigr).
\end{aligned}
\]
Thus the condition in Theorem~\ref{thm:amalgamated-characterization} holds.
Therefore $E=M\bowtie^\varphi JN$ is an annihilator multiplication $R$-module.

For completeness, the annihilators in $C=k[\varepsilon]/(\varepsilon^2)$ are
as follows:
\[
\ann_C(m+\lambda\varepsilon)
=
\begin{cases}
0, & m\ne 0,\\
(\varepsilon), & m=0\text{ and }\lambda\ne 0,\\
C, & m=0\text{ and }\lambda=0.
\end{cases}
\]
Thus the theorem is illustrated by a nonreduced amalgamated algebra with
nonzero square-zero amalgamating ideal. Although $E$ is isomorphic to the free cyclic $R$-module $R$, the example makes
the criterion completely explicit.
\end{example}

\subsection{Square-zero rings and bilinear multiplication}
\label{subsec:squarezero}

The previous coordinate criterion becomes a structural classification
for faithful modules over local rings with square-zero maximal ideal.
For an $R$-module $E$ over a local ring $(R,\mathfrak m,k)$, its
\emph{socle} is $\Soc_R(E)=(0:_E\mathfrak m)$.
A bilinear map $\beta:U\times V\to W$ of vector spaces will be called
\emph{nonsingular} if $\beta(u,v)\ne0$ whenever $u\ne0$ and $v\ne0$.

\begin{theorem}\label{thm:squarezero-bilinear}
Let $(R,\mathfrak m,k)$ be a local ring with
$0\ne\mathfrak m$ and $\mathfrak m^2=0$, and let $E$ be faithful.
Put $W=\Soc_R(E)$ and $T=E/W$.
Then $T$ and $W$ are nonzero $k$-vector spaces, and multiplication
induces a well-defined $k$-bilinear map
\[
 \beta:\mathfrak m\times T\longrightarrow W,
 \qquad (v,x+W)\longmapsto vx.
\]
The following conditions are equivalent:
\begin{enumerate}[label=\textup{(\roman*)}]
\item $E$ is annihilator multiplication;
\item every nonzero element of $E$ has annihilator $0$ or $\mathfrak m$;
\item $\beta$ is nonsingular.
\end{enumerate}
If they hold, $\GammaSD(E)$ consists of exactly two isolated vertices.
\end{theorem}

\begin{proof}
Since $\mathfrak m^2=0$, one has $\mathfrak m E\subseteq W$, so
$T$ and $W$ are killed by $\mathfrak m$.
Faithfulness and $\mathfrak m\ne0$ give $\mathfrak mE\ne0$,
whence $W\ne0$ and $T\ne0$.
Multiplication is unchanged on replacing $x$ by $x+w$ with $w\in W$,
which proves the assertions about $\beta$.

For every nonzero proper ideal $I\subseteq R$,
$\mathfrak m I=0$ and no unit kills $I$, so $\ann_R I=\mathfrak m$.
Thus annihilators of finitely generated ideals are exactly
$R$, $0$, and $\mathfrak m$; all three occur, since a nonzero
principal ideal in $\mathfrak m$ realizes the last one.
Faithfulness now proves the equivalence of (i) and (ii).
A nonzero element of $W$ has annihilator $\mathfrak m$.
For $x\notin W$, every annihilator of $x$ lies in $\mathfrak m$,
and
$\ann_R(x)=\ker(\beta(-,x+W))$.
This proves the equivalence with (iii).

Both annihilators $0$ and $\mathfrak m$ occur, using respectively an
element outside $W$ and a nonzero element of $W$.
Finite intersections introduce no further annihilators for nonzero
finitely generated submodules, and $0+\mathfrak m\ne R$.
\end{proof}

The next estimate applies the classical projective dimension theorem
to the multiplication map. Its linear-algebra content is the usual
Hopf dimension bound for a nonsingular bilinear map; we give the
argument to make the application explicit. We use only the fact that
projective subvarieties of $\mathbb P^N$ with dimensions summing to at
least $N$ must meet; see \cite[Chapter~I, Section~7]{Hartshorne77}.

\begin{theorem}\label{thm:squarezero-bound}
In Theorem~\ref{thm:squarezero-bilinear}, assume that $k$ is
algebraically closed, that $E$ is annihilator multiplication, and that
\[
 r=\dim_k\mathfrak m,\qquad t=\dim_k(E/\Soc_R(E)),\qquad
 w=\dim_k\Soc_R(E)
\]
are finite. Then $w\ge r+t-1$, and consequently
$\ell_R(E)\ge r+2t-1$.
\end{theorem}

\begin{proof}
The linear map $b:\mathfrak m\otimes_k T\to W$ induced by $\beta$
has no nonzero decomposable tensor in its kernel $K$.
If $K=0$, then $w\ge rt\ge r+t-1$.
Otherwise $\mathbb P(K)$ is disjoint from the Segre variety
\[
 \mathbb P^{r-1}\times\mathbb P^{t-1}
 \subseteq\mathbb P(\mathfrak m\otimes_k T)=\mathbb P^{rt-1},
\]
whose dimension is $r+t-2$.
The projective dimension theorem gives
$(\dim_kK-1)+(r+t-2)<rt-1$.
Thus $\operatorname{rank}(b)=rt-\dim_kK\ge r+t-1$, and $w$ is
at least this rank. Finally the exact sequence $0\to W\to E\to T\to0$
gives $\ell_R(E)=w+t$.
\end{proof}

Polynomial multiplication supplies examples attaining this bound.
The accompanying endomorphism calculation proves indecomposability,
so these are not simply direct sums of free or cyclic constructions.
The polynomial bilinear map itself is classical; the point here is its
annihilator multiplication realization and the sharp module properties.

\begin{theorem}\label{thm:sharp-indecomposable-family}
Let $k$ be any field, let $r,t\ge2$, and put
\[
 V=k[z]_{<r},\qquad T=k[z]_{<t},\qquad W=k[z]_{<r+t-1},
\]
where $k[z]_{<a}$ consists of polynomials of degree less than $a$.
Let $R=k\ltimes V$ be the idealization, with multiplication
$(a,v)(b,v')=(ab,av'+bv)$, and define $E=T\oplus W$ by
\[
 (a,v)(u,w)=(au,aw+vu).
\]
Then $E$ is a faithful, indecomposable, nonprojective annihilator
multiplication $R$-module. More precisely,
\[
 \begin{gathered}
 \Soc_R(E)=\mathfrak mE=0\oplus W,\qquad
 \mathfrak m=0\ltimes V,\\
 \mu_R(E)=t,\qquad \ell_R(E)=r+2t-1,\qquad
 \End_R(E)\cong k\ltimes\Hom_k(T,W).
 \end{gathered}
\]
Here $\mu_R(E)$ is the least number of generators, and the ideal
$\Hom_k(T,W)$ in the displayed endomorphism ring has square zero.
Thus the lower bound in Theorem~\ref{thm:squarezero-bound} is attained
for every pair $r,t\ge2$ over every algebraically closed field.
\end{theorem}

\begin{proof}
Polynomial degrees ensure $vu\in W$, and a direct multiplication
shows that the displayed action is associative and unital.
If $(a,v)$ kills $E$, its action on $T$ gives $a=0$, and its action
on $(1,0)$ then gives $v=0$. Hence $E$ is faithful.
Because $k[z]$ is a domain, $vu\ne0$ for nonzero $v\in V$ and
$u\in T$. Thus an element $(u,w)$ with $u\ne0$ has zero
annihilator, whereas $(0,w)\ne0$ has annihilator $\mathfrak m$.
Theorem~\ref{thm:squarezero-bilinear} proves the required property
and the socle equality.
Every monomial of degree at most $r+t-2$ is a product of one of degree
at most $r-1$ and one of degree at most $t-1$.
Consequently $\mathfrak mE=0\oplus W$.
Nakayama's lemma gives $\mu_R(E)=\dim_kT=t$, and dimension gives
$\ell_R(E)=t+(r+t-1)$.

An $R$-endomorphism preserves $0\oplus W$ and therefore has the form
\[
 f(u,w)=(\alpha(u),\gamma(u)+\delta(w)),
\]
where $\alpha:T\to T$, $\gamma:T\to W$, and $\delta:W\to W$
are $k$-linear and satisfy $\delta(vu)=v\alpha(u)$.
Taking $v=1$ gives $\delta|_T=\alpha$.
Taking $v=z$ and $u=z^j$ for $0\le j\le t-2$ gives
$\alpha(z^{j+1})=z\alpha(z^j)$.
Hence $\alpha(z^j)=z^j\alpha(1)$ for $0\le j<t$.
Since $z^{t-1}\alpha(1)$ must belong to $T$, the polynomial
$\alpha(1)$ is a scalar $c\in k$.
It follows that $\alpha=c\operatorname{id}_T$ and, since the products
$vu$ span $W$, $\delta=c\operatorname{id}_W$.
There is no restriction on $\gamma$.
Composition is
$(c,\gamma)(d,\eta)=(cd,c\eta+d\gamma)$, which proves the asserted
ring isomorphism. This ring is local, so it has no idempotents other
than zero and one. A nontrivial decomposition of $E$ would give a
nontrivial idempotent endomorphism. Thus $E$ is indecomposable.

A finitely generated projective module over the local ring $R$ is free.
If $E$ were free, its rank would be $t$ and its length would be
$(r+1)t$. But
$(r+1)t-(r+2t-1)=(r-1)(t-1)>0$.
Therefore $E$ is not projective.
\end{proof}

The cyclic embedding criterion does not imply torsionlessness of the
whole module, even for a finitely generated faithful module. Recall that
an $R$-module is \emph{torsionless} if its homomorphisms to $R$ separate
its elements, equivalently if it embeds in a direct product of copies of $R$.

\begin{corollary}\label{cor:cyclic-not-torsionless}
For the modules $E=T\oplus W$ in
Theorem~\ref{thm:sharp-indecomposable-family}, every cyclic submodule
embeds in $R$, but $E$ is not torsionless. In fact,
\[
 \Hom_R(E,R)\cong\Hom_k(T,V),\qquad
 \bigcap_{f\in\Hom_R(E,R)}\ker f=0\oplus W.
\]
\end{corollary}

\begin{proof}
An element outside $0\oplus W$ generates a copy of $R$.
A nonzero element of $0\oplus W$ generates a copy of $k$, which
embeds in any nonzero one-dimensional subspace of $\mathfrak m$.
This proves the cyclic assertion.

Any $R$-homomorphism $f:E\to R=k\oplus V$ has the form
$f(u,w)=(\alpha(u),\gamma(u)+\delta(w))$, with $k$-linear
maps $\alpha:T\to k$, $\gamma:T\to V$, and $\delta:W\to V$,
satisfying $\delta(vu)=\alpha(u)v$.
For $0\le j\le t-2$, the choices $(v,u)=(1,z^{j+1})$ and
$(z,z^j)$ give
$\alpha(z^{j+1})1=\delta(z^{j+1})=\alpha(z^j)z$.
The linear independence of $1$ and $z$ in $V$ forces
$\alpha(z^j)=\alpha(z^{j+1})=0$; these pairs exhaust a basis of $T$.
Thus $\alpha=0$, and the products $vu$ span $W$, so $\delta=0$.
Conversely, every $\gamma\in\Hom_k(T,V)$ gives such a homomorphism.
These maps separate the elements of $T$, but kill all of $W$, proving
both equalities and the failure of torsionlessness.
\end{proof}

\begin{example}\label{ex:five-dimensional-squarezero}
Taking $r=t=2$ gives $R=k[x,y]/(x,y)^2$ and a five-dimensional
module with basis $u_0,u_1,w_0,w_1,w_2$ and action
\[
 \begin{aligned}
 xu_0&=w_0,& xu_1&=w_1,\\
 yu_0&=w_1,& yu_1&=w_2,
 \end{aligned}
 \qquad xw_i=yw_i=0\quad(0\le i\le2).
\]
This module is faithful, indecomposable, generated by two elements,
and not projective. Its nonzero element annihilators are exactly
$0$ and $(x,y)$. Over an algebraically closed field it has the least
possible length among faithful annihilator multiplication modules
with $\dim_k\mathfrak m=2$ and
$\dim_k(E/\Soc_R(E))=2$.
\end{example}

\section*{Conclusions}

The cyclic embedding criterion separates the defining property from any
finite-generation assumption on the whole module. It yields different
ring classifications for faithful modules and for arbitrary modules, and
it identifies the role of finite annihilator detection in localization,
support, torsion, and graph rigidity.
The reduced finite-minimal-prime theorem and the square-zero bilinear
criterion describe two complementary settings: regular-element torsion
controls the reduced case, whereas nonsingularity of the multiplication
map controls the square-zero case.

The counterexamples delimit these results. Maximal localizations alone
do not give the global property for arbitrary modules over a Noetherian
ring, regular-torsion-freeness alone is insufficient with infinitely many
minimal primes, and a quasi-Frobenius base ring does not make all of its
nonfaithful modules annihilator multiplication.
The indecomposable square-zero family attains the algebraically closed
length bound and shows that the theory includes substantial nonprojective
module constructions.


\begin{thebibliography}{9}

\bibitem{ElKhalfaouiEtAl2021}
R. El Khalfaoui, N. Mahdou, P. Sahandi, and N. Shirmohammadi,
Amalgamated modules along an ideal,
\emph{Commun. Korean Math. Soc.} \textbf{36} (2021), no.~1, 1--10.
doi:10.4134/CKMS.c200064.

\bibitem{KocAMM}
S.~Ko\c{c}, \emph{On annihilator multiplication modules},
J.~Algebra Appl., accepted; preprint
\href{https://arxiv.org/abs/2510.03791v2}{arXiv:2510.03791v2} (2026).

\bibitem{Lam99}
T.~Y. Lam, \emph{Lectures on Modules and Rings},
Graduate Texts in Mathematics, Vol.~189,
Springer-Verlag, New York, 1999.
\href{https://doi.org/10.1007/978-1-4612-0525-8}{doi:10.1007/978-1-4612-0525-8}.

\bibitem{WK24}
F.~Wang and H.~Kim,
\emph{Foundations of Commutative Rings and Their Modules}, 2nd ed.,
Algebra and Applications, Vol.~31,
Springer, Singapore, 2024.
\href{https://doi.org/10.1007/978-981-97-5284-3}{doi:10.1007/978-981-97-5284-3}.


\bibitem{JTK22}
C.~Jayaram, \"U.~Tekir, and S.~Ko\c{c},
\emph{On Baer modules}, Rev.~Uni\'on Mat.~Argentina \textbf{63} (2022),
no.~1, 109--128.
\href{https://doi.org/10.33044/revuma.1741}{doi:10.33044/revuma.1741}.

\bibitem{Storrer69}
H.~H. Storrer, \emph{A note on quasi-Frobenius rings and ring epimorphisms},
Canad.~Math.~Bull. \textbf{12} (1969), no.~3, 287--292.
\href{https://doi.org/10.4153/CMB-1969-036-9}{doi:10.4153/CMB-1969-036-9}.

\bibitem{Hartshorne77}
R.~Hartshorne, \emph{Algebraic Geometry},
Graduate Texts in Mathematics, Vol.~52,
Springer-Verlag, New York--Heidelberg, 1977.
\href{https://doi.org/10.1007/978-1-4757-3849-0}{doi:10.1007/978-1-4757-3849-0}.

\end{thebibliography}
\end{document}